\documentclass[12pt,oneside]{amsart}
\usepackage{graphics,graphicx} 
\usepackage[T1]{fontenc}
\usepackage[utf8]{inputenc}
\usepackage[DIV=11]{typearea}
\usepackage[arrow, matrix, curve]{xy}
\usepackage{amsmath,amsthm,amssymb}
\usepackage{amscd}
\usepackage{mathabx}
\usepackage[pdftitle={Topological freeness for C*-correspondences},pdfauthor={Toke Meier Carlsen, Bartosz Kosma Kwasniewski, and Eduard Ortega},pdfkeywords={Topological freeness; C*-correspondence; Cuntz-Pimsner algebra},pdfsubject={Operator algebra}]{hyperref}

\newenvironment{Proof of}[1]{\emph{Proof of #1.}}{\hfill $\square$\par}

\DeclareMathOperator{\dashind}{-Ind}

\DeclareMathOperator{\Aut}{Aut}

\DeclareMathOperator{\clsp}{\overline{span}}

\DeclareMathOperator{\Irr}{Irr}

\newcommand{\QQ}{\mathcal Q}

\newcommand{\K}{\mathcal K}

\newcommand{\X}{\widehat X}

\newcommand{\F}{\mathcal F}

\newcommand{\sal}{\widehat \alpha}

\DeclareMathOperator{\Int}{Int}
\newcommand{\LL}{\mathcal{L}}
\newcommand{\al}{\alpha}

\newcommand{\OO}{\mathcal O}

\newcommand{\B}{\mathcal B}

\newcommand{\SA}{\widehat{A}}
\newcommand{\SB}{\widehat{B}}
\newcommand{\SX}{\widehat{X}}
\newcommand{\SI}{\widehat{I}}

\newcommand{\C}{\mathbb C}

\newcommand{\Z}{\mathbb Z}
\newcommand{\N}{\mathbb N}
\newcommand{\T}{\mathbb T}
\newcommand{\TT}{\mathcal T}

\newcommand{\supp}{\textrm{supp}\,}

\newtheorem{thm}{Theorem}[section]\newtheorem{lem}[thm]{Lemma} 
\newtheorem{prop}[thm]{Proposition} 
\newtheorem{cor}[thm]{Corollary}
\theoremstyle{definition} 
\newtheorem{defn}[thm]{Definition}
\newtheorem{ex}[thm]{Example}
\newtheorem{rem}[thm]{Remark}

\title{Topological freeness for $C^*$-correspondences}

\author[T.M. Carlsen]{Toke Meier Carlsen} 
\address{Department of Science and Technology\\University of the Faroe Islands\\
Vestara Bryggja 15\\ FO-100 T\'orshavn\\the Faroe Islands}
\email{toke.carlsen@gmail.com}

\author[B. K. Kwa\'sniewski]{Bartosz Kosma Kwa\'sniewski} 
\address{Institute of Mathematics, University of Bia\l ystok\\
ul. K. Cio\l kowskiego 1M, 15-245, Bia\l ystok, Poland}
\email{bartoszk@math.uwb.edu.pl}

\author[E. Ortega]{Eduard Ortega} 
\address{Department of Mathematical Sciences\\NTNU\\NO-7491 Trondheim\\Norway}
\email{eduard.ortega@ntnu.no}
\keywords{Topological freeness; $C^*$-correspondence; Cuntz--Pimsner algebra}

\begin{document}

\begin{abstract} 
We study conditions that ensure uniqueness theorems of Cuntz--Krieger type for relative Cuntz--Pimsner algebras $\OO(J,X)$ associated to a $C^*$-correspondence $X$ over a $C^*$-algebra $A$. We give general sufficient conditions phrased in terms of a multivalued map $\widehat{X}$ acting on the spectrum $\SA$ of $A$. When $X(J)$ is of Type I we construct a directed graph dual to $X$ and prove a uniqueness theorem using this graph. When $X(J)$ is liminal, we show that topological freeness of this graph is equivalent to the uniqueness property for $\OO(J,X)$, as well as to an algebraic condition which we call $J$-acyclicity of $X$.

As an application we improve the Fowler--Raeburn uniqueness theorem for the Toeplitz algebra $\TT_X$. We give new simplicity criteria for $\OO_X$. We generalize and enhance uniqueness results for relative quiver $C^*$-algebras of Muhly and Tomforde. We also discuss applications to crossed products by endomorphisms. 
\end{abstract}

\maketitle

\setcounter{tocdepth}{1}
\section{Introduction} 

The condition, known as \emph{topological freeness}, was probably first formulated in \cite{OD}, in the context of crossed products of $\Z$-actions (see \cite[pp. 225, 226]{Anton_Lebed} or \cite{kwa} for more history of this notion). It implies that any faithful covariant representation of the system generates an isomorphic copy of the crossed product. A similar uniqueness property for Cuntz-- Krieger algebras was identified in \cite{CK}. Since then, uniqueness results concerning various generalizations of Cuntz--Krieger algebras are typically called Cuntz--Krieger uniqueness theorems. This concerns, for instance, graph $C^*$-algebras \cite{Raeburn}; Exel--Laca algebras \cite{Exel_Laca}, Matsumoto algebras \cite{Matsumoto1}, \cite{Matsumoto2}, ultragraph algebras \cite{Tomforde}, \cite{Gon_Li-Royer-Rae}, $C^*$-algebras of labelled graphs \cite{Bates_Pask}, $C^*$-algebras of topological graphs \cite{ka1}, quivers $C^*$-algebras \cite{mt}, Exel's crossed products \cite{exel_vershik}, \cite{er}, \cite{CS}, \cite{brv}, crossed products by endomorphisms \cite{kwa-lebedev}, \cite{kwa-rever}, \cite{kwa-endo}, and numerous semigroup generalizations of these constructions. Importantly, Katsura's uniqueness theorem \cite{ka1} unified the combinatorial condition for graphs, called condition (L), and topological freeness for a homeomorphism.

Nowadays a standard model for all of the above mentioned constructions is a Cuntz--Pimsner algebra, introduced by Pimsner \cite{Pim} and further developed by Katsura \cite{ka2, katsura}. A Cuntz--Pimsner algebra is a universal $C^*$-algebra $\OO_X$ associated to a $C^*$-correspondence $X$ over a $C^*$-algebra $A$. It is thus desirable to have uniqueness theorems for general Cuntz-Pimsner algebras $\OO_X$, and some partial results in this direction have been obtained. For instance, when $X$ is a Hilbert bimodule, then $\OO_X$ coincides with the crossed product $A\rtimes_X \Z$ introduced in \cite{AEE}, and $X$ induces a partial homeomorphism $\X$ on the spectrum $\SA$. The second named author proved in \cite{kwa} that topologically freeness of $\X$ implies the uniqueness property for $A\rtimes_X \Z$. By \cite{KM} the converse implication holds when $A$ contains an essential ideal which is separable or of Type I. The result of \cite{kwa} was adapted in \cite{kwa-szym} to $C^*$-correspondences whose left action is injective and by compacts. The authors of \cite{kwa-szym} introduced a multivalued map $\X$ dual to a regular $C^*$-correspondence $X$ and proved a uniqueness theorem for (the semigroup version of) $\OO_X$ under the condition they called \emph{topological aperiodicity}. This condition seems to be well-suited for Hilbert bimodules and $C^*$-correspondences associated to transfer operators (cf. Example~\ref{Exel's crossed products} below). However, already for $C^*$-correspondences coming from graphs it is too strong as the multivalued map $\X$ does not capture multiplicities of edges. The original motivation behind the present paper was the need of developing a theory that does not have this deficiency.

In the present article we consider general $C^*$-correspondences and prove uniqueness theorems for relative Cuntz--Pimsner algebras $\OO(J,X)$ where $J$ is an ideal contained in Katsura's ideal $J_X$. These relative Cuntz--Pimsner algebras $\OO(J,X)$ were introduced by Muhly and Solel in \cite{MuSo1998}. They contain as extreme cases the Cuntz--Pimsner algebra $\OO_X$, for $J=J_X$, and the Toeplitz algebra $\TT_X$, when $J=\{0\}$. This allows us to cover another line of uniqueness theorems inspired by Coburn's theorem \cite{coburn}. The latter type of results, proved for various generalized Toeplitz algebras, play a fundamental role in the theory of $C^*$-algebras associated to semigroups and semigroup actions \cite{N}, \cite{FR}, \cite{Fow-Rae}, \cite{F99}, \cite{kwa-larI}, \cite{kwa-larII}. An important contribution in this line of research is Fowler--Raeburn's version of the Coburn theorem \cite[Theorem 2.1]{Fow-Rae} proved for the Teoplitz algebra $\TT_X$. As the authors of \cite{Fow-Rae} show the geometric condition they introduce is necessary and sufficient for the uniqueness property of $\TT_X$ when the left action of $A$ on $X$ is by generalized compacts. In general, it is only sufficient. Thus, for instance, in order to deduce simplicity of the Cuntz algebra $\OO_\infty$ one needs a stronger result. As an upshot of our uniqueness theorem we get an improvement of the Fowler--Raeburn theorem. Namely, we show that a representation of $X$ generates a copy of the Toeplitz algebra $\TT_X$ if and only if its ideal of covariance is $\{0\}$.

Our uniqueness theorem (Theorem~\ref{uniqueness theorem}) consists in fact of two statements. Firstly, we extend the construction of dual multivalued maps from \cite{kwa-szym} to arbitrary $C^*$-corre\-sponden\-ces $X$. We show that a pair $(X,J)$ has the uniqueness property (Definition~\ref{defn:uniqueness_property}) whenever the dual multivalued map $\X$ satisfies on $\widehat{J}$ a weaker form of the condition introduced in \cite{kwa-szym}. We call this condition \emph{weak topological aperiodicity} (Definition~\ref{def:weakly_top_aperiod}). It works better than topological aperiodicity, for instance, with $C^*$-correspondences coming from endomorphisms (cf. Example~\ref{endomorphisms of C(V)-algebras}). It can be used to characterise simplicity of $\OO_X$ when the left action is not injective (Proposition~\ref{simple:non-injective}). Moreover, when $J=\{0\}$ this condition is void and hence the pair $(X,\{0\})$ always has the uniqueness property. This gives the aforementioned improvement of the Fowler--Raeburn theorem (see Theorem~\ref{theorem for Toeplitz algebras}).

Secondly, in order to get sharper results we adjust the construction of the dual multivalued map. We take into account multiplicities of the corresponding subrepresentations and construct a directed graph $E_X=(\SA, E^1_X,r,s)$ dual to $X$. This can be done whenever the multiplicity theory is available, so for instance in the Type I case. Moreover, in order to detect the uniqueness property it suffices to consider the restriction of $X$ to the ideal $K=J+X(J)$ where $X(J)$ is the ideal generated by $\langle X, JX \rangle_A$. Therefore we assume that $X(J)$ is of Type I. Then $K$ is of Type I and we may consider the graph $E_{KXK}=(\widehat{K}, E^1_{KXK},r,s)$ dual to the restricted $C^*$-correspondence $KXK$. It seems that in this generality Katsura's condition needs a slight strengthening. Therefore we introduce a notion of \emph{strong topological freeness} (Definition~\ref{top_free_graph1}). If $X(J)$ is liminal (so for instance, when it is commutative), then strong topological freeness is equivalent to Katsura's topological freeness (Proposition~\ref{prop:strong_vs_normal}). We prove that if $E_{KXK}$ is strongly topologically free, then $(X,J)$ has the uniqueness property (Theorem~\ref{uniqueness theorem} (A2)).

We get the nicest results in the case when $X(J)$ is liminal. Then we are able to show that topological freeness is necessary for the uniqueness property for $\OO(J,X)$, and in addition we characterise it using an algebraic condition that we call \emph{$J$-acyclicity} (Definition~\ref{def:cycling}). This condition was identified in a purely algebraic setting in \cite{Carlsen_Ortega_Pardo}. In particular, by Theorem~\ref{thm:acyclicity} we have:

\begin{thm}
Let $X$ be a $C^*$-correspondence over $A$ and let $J$ be an ideal in $J_X$. If the ideal $X(J)$ is liminal then the following conditions are equivalent:
\begin{enumerate}
\item for every injective representation $(\psi_0,\psi_1)$ of $X$ whose ideal of covariance is $J$ the map $\psi_0\rtimes_{J}\psi_1: \OO(J,X) \to C^* (\psi_0(A)\cup \psi_1(A))$ is an isomorphism;
\item $X$ is $J$-acyclic, i.e. there are no non-zero positively invariant ideals $I$ in $J$ such that $(IX)^{\otimes n}\cong I$ for some $n>0$;
\item the corresponding dual graph is topologically free on $\widehat{J}$, i.e. the set of base points of cycles that have no entrances and whose base points are contained in $\widehat{J}$ has empty interior. 
\end{enumerate}
\end{thm}

Our results have strong potential for numerous applications, including description of ideal lattice, cf. \cite{ka3}, \cite{kwa-doplicher}, or pure infiniteness criteria for $\OO(X,J)$, cf. \cite{KM}. A thorough discussion of such applications is beyond the scope of the present paper. As an illustration we discuss here simplicity criteria for $\OO_X$ (Subsection~\ref{sub:Simple Cuntz-Pimsner algebras}). We also generalize and improve uniqueness results for relative quiver $C^*$-algebras (Subsection~\ref{sub:Topological correspondences}) and crossed products by endomorphisms (Subsection~\ref{sub:Crossed products by endomorphisms}).

The paper is organized as follows. In Section~\ref{Preliminaries} we present the terminology and some well known facts that will be used along the paper. In Section~\ref{Section 3} we discuss and establish relationships between various non-triviality conditions for multivalued maps and directed graphs. Section~\ref{sect:Multivalued maps} presents the construction and properties of the multivalued map dual to a $C^*$-correspondence $X$. In Section~\ref{sec:Dual graphs} we construct graphs dual to $C^*$-correspondences and prove certain technical statements that we need in the proof of our main result. In Section~\ref{sec:The uniqueness property} we discuss some general characterizations and criteria for the uniqueness property for the pair $(X,J)$. Section~\ref{sec:The uniqueness theorem} contains the proof of our general uniqueness theorem. In Section~\ref{sec:Conditions necessary for uniqueness property} we introduce $J$-acyclicity of $X$ and examine necessity of this condition and topological freeness for the uniqueness property. Finally, in Section~\ref{sec:Applications} we discuss applications of our results to $\TT_X$, simplicity of $\OO_X$, relative quiver $C^*$-algebras and crossed products by endomorphisms.

\subsection{Acknowledgements}

The research leading to these results was supported by the project ``Operator algebras and single operators via dynamical properties of dual objects'' funded by the European Commission's Seventh Framework Programme\linebreak (FP7/2007-2013) under grant agreement number 621724.

The authors would like to thank the referee for carefully reading our manuscript and for giving such constructive comments.

\section{Preliminaries}\label{Preliminaries}

In this section, we fix terminology and recall some known facts, that we need in the sequel. We stress that for actions $\gamma\colon A\times B\to C$ such as multiplications, inner products, etc., we use the notation: $\gamma(A,B)=\clsp\{\gamma(a,b) :a\in A, b\in B\}$.

\subsection{Multivalued maps}

We adopt the same conventions concerning multivalued maps as in \cite{kwa-szym}. Thus a {\em multivalued mapping} from a set $V$ to a set $W$ is by definition a mapping from $V$ to $2^W$ (the family of all subsets of $W$). We denote such a multivalued mapping by $f:V\to W$. Also, we identify the usual (single-valued) mappings with multivalued mappings taking values in singletons. We define the image of $V'\subseteq V$ under $f$ by $f(V'):=\bigcup_{x\in V'} f(x)$. The \emph{range of $f$} is $f(V)$. We say that $f$ is \emph{surjective} if $f(V)=W$. We define the preimage of $W'\subseteq W$ to be the set $ f^{-1}(W'):=\{v\in V: f(v)\cap W' \neq \emptyset \}. $ In particular, we call $ f^{-1}(W)=\{v\in V: f(v)\neq \emptyset\}$ the \emph{domain} of $f$. If $V$ and $W$ are topological spaces, we say that a multivalued map $f:V\to W$ is \emph{continuous} if $f^{-1}(U)$ is open for every open subset $U$ of $W$. 

The \emph{composition} of two multivalued maps $f:V\to W$ and $g:W\to Z$ is the multivalued map $g\circ f:V\to Z$ given by $(g\circ f)(x):=\bigcup_{y \in f(x)} g(y)$. The class of sets and multivalued maps with the above composition form a category (the composition is associative and identity maps are identity morphisms). A multivalued map is invertible in this category if and only if it is a bijective single valued map.

Given $V'\subseteq V$ we define the restriction $f_{V'}:V'\to W$ of $f:V\to W$ in the obvious way. If $W'\subseteq W$ we define $_{W'}f:V\to W'$ by the formula $ _{W'}f(v)=f(v)\cap W'$, $v\in V. $ We also define ${_{W'}}f_{V'}:V'\to W'$ by $_{W'}f_{V'}(v):={_{W'}f}(v)$, for $v\in V'$.

\subsection{$C^*$-correspondences and Hilbert bimodules.} 

We adopt standard notations and definitions of objects related to Hilbert $C^*$-modules, cf. \cite{lance} and \cite{morita}. Thus given right Hilbert modules $X$ and $Y$ over a $C^*$-algebra $A$ we denote by $\LL(X,Y)$ the space of adjointable operators from $X$ to $Y$. Given $x\in X$ and $y\in Y$ we define $\Theta_{y,x}\in \LL(X,Y)$ by $\Theta_{y,x}(z):=y\langle x,z\rangle_A$, for $z\in X$. The elements of $\K(X,Y):=\clsp\{\Theta_{y,x}: x\in X,y\in Y\}$ are called \emph{generalized compact operators}. In particular, $\LL(X):=\LL(X,X)$ is a $C^*$-algebra and $\K(X):=\K(X,X)$ is an essential ideal of $\LL(X)$.

If $I$ is an ideal of $A$ (which will always be closed and two-sided) then $XI:=\{xa:x\in X,\,a\in I\}$ is both a Hilbert $A$-submodule of $X$ and a Hilbert $I$-module \cite[Proposition 1.3]{ka3}. In particular, we can treat $\K(XI)$ as an ideal of $\K(X)$.

A representation of a Hilbert $A$-module $X$ on a $C^*$-algebra $C$ is a pair $\psi=(\psi_0,\psi_1)$ consisting of a $*$-homomorphism $\psi_0:A\to C$ and a linear map $\psi:X\to C$ such that
$$ \psi_1(x) \psi_0( a)=\psi_1(x \cdot a),
\qquad \psi_1(x)^*\psi_1(y)=\psi_0(\langle x, y\rangle_A),\quad a\in A,\, x\in X.$$ 
If $C=\B(H)$, where $H$ is a Hilbert space, we say that $\psi$ is a representation on $H$. Any representation $\psi$ of $X$ induces a representation $\psi^{(1)}:\K(X)\to C$ where 
\begin{equation}\label{eq:Pimnser representation} 
\psi^{(1)}(\Theta_{x,y})=\psi_1(x)\psi_1(y)^*\qquad\text{for all } x, y\in X. 
\end{equation}

Given $C^*$-algebras $A$ and $B$, a \emph{$C^*$-correspondence from $A$ to $B$}, written $X:A\to B$, is a right Hilbert $B$-module $X$ together with a $*$-homomorphism $\phi_X:A\rightarrow \LL(X)$ called \emph{left action of $A$ on $X$}. We will write $a\cdot x:=\phi_X(a)(x)$ for $a\in A$ and $x\in X$. In case that $A=B$ we say that $X$ is a \emph{$C^*$-correspondence over $A$}.

Let $A$ and $B$ be $C^*$-algebras. A Hilbert \emph{$A$-$B$-bimodule} $M$ is both a right Hilbert $B$-module and a left Hilbert $A$-module such that the respective $A$- and $B$-valued inner products, ${}_A\langle\cdot,\cdot\rangle$ and $\langle\cdot,\cdot\rangle_B$ satisfy $ {}_A\langle x,y\rangle z=x\langle y, z\rangle_B, $ for every $x,y,z\in M$. Then $M$ is a $C^*$-correspondence from $A$ to $B$, cf. \cite[Section 3.3]{ka2}, \cite[Proposition 1.11]{kwa-doplicher}. An \emph{equivalence $A$-$B$-bimodule} is a Hilbert $A$-$B$-bimodule which is full on the left and right, that is ${}_A\langle M,M\rangle=A$ and $\langle M, M\rangle_B=B$. If such a bimodule exists we say that $A$ and $B$ are \emph{Morita(-Rieffel) equivalent}. We write $A\sim_M B$ when we want to indicate that $M$ is an equivalence $A$-$B$-bimodule. Given an $A$-$B$-bimodule $M$, we denote by $M^{*}$ the $B$-$A$-bimodule obtained from $M$ by exchanging the roles of left and right actions ($M$ and $M^*$ are anti-isomorphic as linear spaces).

Given $C^*$-algebras $A$, $B$, $C$, and $C^*$-correspondences $X:A\to B$ and $Y:B\to C$, we define a $C^*$-correspondence $X\otimes_B Y:A\to C$ in the following way. Let $X\odot Y$ denote the quotient of the algebraic tensor product of $X$ and $Y$ by the subspace generated by $xb\otimes y-x\otimes (b\cdot y)$ for $x\in X$, $y\in Y$ and $b\in B$. There is a $C$-valued inner product on $X\odot Y$ determined by the formula 
$$ \langle x_1\otimes y_1 , x_2\otimes y_2\rangle_C:=\langle y_1, \phi_Y(\langle x_1, x_2\rangle_B) y_2\rangle_C, $$ 
for $x_1,x_2\in X$ and $y_1,y_2\in Y$. The \emph{(inner) tensor product} of $X$ and $Y$, denoted by $X\otimes_B Y$, is the completion of $X\odot Y$ with respect to the norm coming from the $C$-valued inner product defined above. Then $X\otimes_B Y$ is a $C^*$-correspondence from $A$ to $C$ with left action given by $\phi_{X\otimes_B Y}:=\phi_X\otimes 1_Y$ where $1_Y$ is the identity map on $Y$. In order to lighten the notation, we will often write $X\otimes Y$ instead of $X\otimes_B Y$.

If $M$ is an equivalence $A$-$B$-bimodule then we have $C^*$-correspondence isomorphisms $m_A:M\otimes_BM^*\rightarrow A$ and $m_B:M^*\otimes_AM\rightarrow B$ given by $x^*\otimes y\mapsto \langle x,y\rangle_B$ and $x\otimes y^*\mapsto {}_A\langle x,y\rangle$, where we treat $A$ and $B$ as trivial $C^*$-correspondences with the structure inherited from $C^*$-algebraic operations.

The correspondences $X$ and $Y$ over $A$ and $B$ respectively are \emph{Morita equivalent}, see \cite[Definition 2.1]{ms}, if there is an equivalence $A$-$B$-bimodule $M$ and an $A$-$B$-correspondence isomorphism $W$ from $M\otimes_B Y$ onto $X\otimes_A M$. This is equivalent to saying that there is an equivalence $A$-$B$-bimodule $M$ such that we have a correspondence isomorphism $M\otimes_B\otimes Y\otimes _B M^* \cong X$. If this holds we write $X\sim_M Y$.

\subsection{Cuntz--Pimsner algebras}

Let us fix a $C^*$-correspondence $X$ over a $C^*$-algebra $A$. For each $n=0,1,2,\ldots$ we denote by $X^{\otimes n}$ the $C^*$-correspondence given by the $n$-fold tensor product $X\otimes \cdots \otimes X$ ($X^{\otimes 0}:=A$). Relative Cuntz--Pimsner algebras \cite{MuSo1998}, could be viewed as crossed products associated to the product system $\{X^{\otimes n}\}_{n\in \N}$ or as $C^*$-algebras associated to the ideal $\{\K(X^{\otimes n},X^{\otimes m})\}_{n,m\in \N}$ in a right tensor $C^*$-precategory $\{\LL(X^{\otimes n},X^{\otimes m})\}_{n,m\in \N}$, see \cite{kwa-doplicher}. We define them in terms of universal representations \cite{fmr}.

\begin{defn} 
A representation of a $C^*$-correspondence $X$ over a $C^*$-algebra $A$ is a representation $\psi=(\psi_0,\psi_1)$ of the Hilbert $A$-module such that $ \psi_1(a\cdot x)=\psi_0(a)\psi_1(x)$, $a\in A$, $x\in X$. The $C^*$-algebra generated by $\psi_0(A)\cup \psi_1(X)$ is denoted by $C^*(\psi)$.
\end{defn}

Let $\psi=(\psi_0,\psi_1)$ be a representation of a $C^*$-correspondence $X$ on a $C^*$-algebra $B$. Then $\psi_1$ is automatically contractive map. If $\psi_0$ is injective, then $\psi_1$ is isometric and we say that $\psi$ is \emph{injective}. For each $n>0$ there is a unique representation $(\psi_0,\psi_n)$ of $X^{\otimes n}$ where 
$$ 
\psi_n(x_1\otimes x_2\otimes\cdots\otimes x_n):=\psi_1(x_1)\psi_1(x_2)\dots \psi_1(x_n) 
$$ 
for all $x_1,...,x_n\in X$. Thus we also have a representation $\psi^{(n)}:\K(X^{\otimes n})\to B$ of the $C^*$-algebra $\K(X^{\otimes n})$, cf. \eqref{eq:Pimnser representation}. If $\psi$ is injective, then $\psi^{(n)}$ is also injective. If $\psi$ is a representation on a Hilbert space $H$, then the above maps $\psi^{(n)}$ extend uniquely to representations $\overline{\psi^{(n)}}:\LL(X^{\otimes n})\to \B(H)$ such that $\overline{\psi^{(n)}}(T)|_{(\psi_n(X^{\otimes n})H)^\bot}\equiv 0$. In the sequel we will need the following: 

\begin{lem}[Lemma 2.5 in \cite{Fow-Rae}]\label{Fowler-Reaburn lemma} 
Suppose $X$ is a $C^*$-correspondence over a $C^*$-algebra $A$ and that $\psi=(\psi_0,\psi_1)$ is a representation of $X$ on a Hilbert space $H$. For each $n>0$, let $\K_n:=\psi^{(n)}(\K(X^{\otimes n}))$ and let $P_n$ be the orthogonal projection of $H$ onto $\K_nH=\psi_n(X^{\otimes n})H$. Then $P_1\geq P_2 \geq ...$, $P_n\in (\K_n)'$, and for $x\in X^{\otimes k}$ and $k \geq 0$ we have 
\begin{equation}\label{essential projections versus fibers} 
\psi_k(x)P_n=P_{n+k}\psi_{k}(x). 
\end{equation} 
\end{lem} 

One associates to $X$ the following two ideals in \(A\): 
$$ J(X):=\phi_{X}^{-1}(\K(X))\qquad\text{and}\qquad J_X:=J(X) \cap
(\ker\phi_{X})^\bot. $$ 
Their significance is indicated in the following lemma.

\begin{lem}\label{lemma on tensoring regular correspondences} 
Let $X$ and $Y$ be two $C^*$-correspondences over $A$. 
\begin{itemize} 
\item[(i)] If $T\in \K(YJ(X))$, then $T \otimes 1_X \in \K(Y\otimes X)$. 
\item[(ii)] If $T\in \K(Y(\ker\phi_{X})^\bot)$, then $\|T\|=\|T \otimes 1_X\|$. 
\end{itemize}
\end{lem}

\begin{proof}
See for instance \cite[Proposition 4.7]{lance} and \cite[Lemma 1.9 i)]{kwa-doplicher}.
\end{proof}

\begin{defn}
Let $X$ be a $C^*$-correspondence over a $C^*$-algebra $A$ and let $J$ be an ideal in $J(X)$. We say that a representation $\psi$ of $X$ is $J$-\emph{covariant} if 
$$
J\subseteq I_\psi:=\{a\in J(X): \psi_0(a)=\psi^{(1)}(\phi_X(a))\}.
$$
We call $I_\psi$ the \emph{ideal of covariance} for $\psi$ (it is an ideal in $J(X)$). We denote by $j=(j_A,j_X)$ the universal $J$-covariant representation of $X$, and we call the $C^*$-algebra $\OO(J,X):=C^*(j_A,j_X)$ the \emph{relative Cuntz--Pimsner algebra} determined by $J$. In particular, if $\psi$ is a $J$-covariant representation of $X$ the maps
$$
j_A(a)\longmapsto \psi_0(a), \qquad j_X(x) \longmapsto \psi_1(x), \qquad a\in A,\, x\in X
$$
define an epimorphism $\psi\rtimes_J X: \OO(J,X) \to C^*(\psi)$. The (unrelative) \emph{Cuntz--Pimsner} is $\OO_X:=\OO(J_X,X)$.
\end{defn} 

\begin{rem}
The universal representation $(j_A,j_X)$ is injective if and only if $J\subseteq J_X$, see \cite[Proposition 2.21]{MuSo1998} and \cite[Proposition 3.3]{katsura}, or \cite[Corollary 4.15]{kwa-doplicher}. Moreover, by passing to a quotient $C^*$-correspondence one may reduce the general case to the case when $J\subseteq J_X$, see \cite{kwa-lebedev} or \cite[Theorem 6.23]{kwa-doplicher}.
\end{rem}

The $C^*$-algebra $\OO(J,X)$ is equipped with the \emph{gauge circle action} $\gamma:\T \to \Aut(\OO(J,X))$ determined by $\gamma_z(j_A(a))=j_A(a)$ and $\gamma_z(j_X(x))=zj_X(x)$, for $a\in A$, $x\in X$ and $z\in \T$. We say that an ideal in $\OO(J,X)$ is \emph{gauge invariant} if it is invariant with respect to this action.

We recall that an image and preimage of an ideal $I$ in $A$ with respect to the $C^*$-correspondence $X$ are defined as follows:
\begin{align*}
X(I) & :=\langle X, \phi_X(I)X\rangle_A=\clsp\{\langle x,a\cdot y\rangle_A: a\in I,\,\, x,y\in X \},\\ 
X^{-1}(I) & :=\{a\in A: \langle x,a\cdot y\rangle_A \in I \textrm{ for all } x,y\in X \}. 
\end{align*}
Both $X(I)$ and $X^{-1}(I)$ are ideals in $A$, cf. \cite{ka3}. If $X(I)\subseteq I$, then the ideal $I$ is said to be \emph{positively $X$-invariant}, \cite[Definition 4.8]{ka3}.

Given a positively $X$-invariant ideal $I$ of $A$, the quotient space $X/XI$ is naturally equipped with the structure of a $C^*$-correspondence over $A/I$. For any positively $X$-invariant ideal $I$ we put
$$
J(I):=\{a\in A: \phi_{X/XI}(a+I)\in\K(X/XI)\,, aX^{-1}(I)\subseteq I \}.
$$
A \emph{$T$-pair} of $X$ is a pair $(I,I')$ of ideals $I,I'$ of $A$ such that $I$ is positively $X$-invariant and $I\subseteq I'\subseteq J(I)$ \cite[Definition 5.6]{ka3}. Given an ideal $J$ of $J_X$ we have a bijective correspondence between the gauge invariant ideals of the relative Cuntz--Pimsner algebra $\OO(J,X)$ and the $T$-pairs $(I,I')$ of $X$ satisfying $J\subseteq I'$, see \cite[Proposition 11.9]{ka3}.

\subsection{Induced and irreducible representations.} 

Let $A$ be a $C^*$-algebra. If $\pi$ is a representation of $A$ we will usually denote the underlying Hilbert space by $H_\pi$, so that we have 	 $\pi:A\rightarrow \mathcal{B}(H_\pi)$. Given a representation $\pi$ of $A$ we denote by $[\pi]$ the class of representations of $A$ that are unitary equivalent to $\pi$. We denote by $\Irr(A)$ the class of all non-zero irreducible representations of $A$. The \emph{spectrum of $A$} is the set $\SA$ of unitary equivalence classes of representations from $\Irr(A)$, equipped with the standard Jacobson topology.

Let $B\subseteq A$ be a sub-$C^*$-algebra of $A$, and let $\pi$ and $\rho$ be representations of $B$ and $A$, respectively. Then $\rho$ is an \emph{extension} of $\pi$, denoted by $\pi\leq \rho$, if $H_\pi$ is a subspace of $H_\rho$, and $\pi(b)h=\rho(b)h$ for every $b\in B$ and $h\in H_\pi$. Given a representation $\pi$ of $B$ there exists a (not necessarily unique) representation $\rho$ of $A$ with $\pi\leq \rho$ (see for example \cite[Proposition 2.10.2]{Dixmier}). Moreover, if $\pi$ is irreducible, then $\rho$ can be chosen to be irreducible. If $I$ is an ideal of $A$, then the map $[\pi]\to[\pi|_I]$ defines a homeomorphism between $\{[\pi]:\,\pi\in\Irr(A),\,\pi(I)\neq 0 \}$ and $\SI$ (see for example \cite[Proposition 3.2.1]{Dixmier}). We will use this map to identify $\SI$ with the open subset $\{[\pi]:\,\pi\in\Irr(A),\,\pi(I)\neq 0 \}$ of $\SA$.

Let $X$ be a right Hilbert $A$-module and let $\pi:A\rightarrow\mathcal{B}(H_\pi)$ be a representation. We may view $X$ as a $C^*$-correspondence from $\LL(X)$ to $A$, and $H_\pi$ as a $C^*$-correspondence from $A$ to $\C$. The corresponding tensor product, which we denote by $X\otimes_\pi H_\pi$ is a $C^*$-correspondence from $\LL(X)$ to $\C$. In other words, $X\otimes_\pi H_\pi$ is a Hilbert space and we have the representation
$$X\dashind (\pi):\LL(X)\rightarrow \mathcal{B}(X\otimes_\pi H_\pi)\qquad\text{where} \qquad X\dashind(\pi)(a) (x\otimes_\pi h)= (ax)\otimes_\pi h \,,$$
for every $a\in \LL(X)$ and $x\otimes_\pi h\in X\otimes_\pi H_\pi$. This defines the inducing functor $X\dashind =X\dashind_{A}^{\LL(X)}$, cf. \cite{morita}. Since $X$ is an equivalence $\K(X)$-$\langle X,X\rangle$-bimodule \cite[Proposition 3.8]{morita}, the map $[\pi]\mapsto [X\dashind (\pi)]$ restricts to a homeomorphism $[X\dashind]$ from $\widehat{\langle X,X\rangle}$ onto $\widehat{\K(X)}$ (see for example \cite[Corollary 3.33]{morita}). Since we identify $\widehat{\langle X,X\rangle}_A$ and $\widehat{\K(X)}$ with open subsets of $\SA$ and $\widehat{\LL(X)}$ respectively, we may view this map as a partial homeomorphism from $\SA$ to $\widehat{\LL(X)}$.

\begin{defn}\label{def:dual_partial_homeomorphism}
Let $X$ be a right Hilbert $A$-module. The \emph{map dual to $X$} is the partial homeomorphism $[X\dashind]:\SA\rightarrow \widehat{\LL(X)}$ constructed above. 
\end{defn}

Let $X^*$ be the Hilbert $\K(X)$-module adjoint to \(X\) (see for example \cite[page 49]{morita}). Then $X^*$ is naturally an equivalence $\langle X,X\rangle_A$-$\K(X)$-bimodule, defining a partial homeomorphism $[X^*\dashind]:\widehat{\LL(X)}\rightarrow \SA$ with domain $\widehat{\K(X)}$ and range $\widehat{\langle X,X\rangle}_A$, such that the restriction ${}_{\widehat{\K(X)}}[X^*\dashind]_{\widehat{\langle X,X\rangle}_A}$ is the inverse of ${}_{\widehat{\langle X,X\rangle}_A}[X\dashind]_{\widehat{\K(X)}}$ (cf. \cite[Theorem 3.29]{morita}. We therefore often write $[X\dashind]^{-1}$ for $[X^*\dashind]$.

\begin{lem}\label{representation structure lemma2}
Suppose that $\psi=(\psi_0,\psi_1)$ is a representation of a $C^*$-correspondence $X$ over $A$ on a Hilbert space $H$. Let $\pi\leq \psi_0$ be a representation of $A$. The map $\psi_1(x)h\mapsto x\otimes_{\pi}h$ extends to a unitary $U:\psi_1(X)H_\pi\to X\otimes_{\pi} H_{\pi}$ such that 
$$
\overline{\psi^{(1)}}(T)|_{\psi_1(X)H_\pi}= U^* ( X\dashind (\pi)(T)) U, \qquad T\in \LL(X).
$$
In particular, $ \LL(X) \ni T \to \overline{\psi^{(1)}}(T)|_{\psi_1(X)H_\pi}$, is a representation which is equivalent to $X\dashind (\pi)$. 
\end{lem}

\begin{proof} 
This is straightforward, see for instance \cite[Lemma 1.3]{kwa} or \cite[Lemma 4.10]{kwa-szym}.
\end{proof}

\section{Non-triviality conditions for multivalued maps and directed graphs.} \label{Section 3}

The aim of the present section is to clarify relationships between various nontriviality conditions for multivalued maps and graphs. The notion of topological aperiodicity was introduced in the context of semigroups of multivalued maps in \cite[Definition 5.3]{kwa-szym}. We recall the corresponding condition for a single multivalued map, see \cite[Proposition 5.5.iii)]{kwa-szym}.

\begin{defn} 
We say that a multivalued map $f:V\to V$ defined on a topological space is \emph{topologically aperiodic} if for every $n>0$ the set $\{ v\in V : v\in f^n(v)\} $ has empty interior in $V$ (here $f^n$ stands for $n$-times composition of the multivalued map $f$).
\end{defn}

Multivalued maps can be viewed as directed graphs without multiple edges. More specifically, by a \emph{graph from $V$ to $W$} we mean a triple $E=(E^1,s,r)$ where $E^1$ is a set, whose elements are called edges, and $s:E^1\to V$ and $r:E^1\to W$ are maps, indicating a source and a range of an edge. Sets $s(E^1)$ and $r(E^1)$ are called, respectively, the \emph{domain} and the \emph{range} of the graph $E$. We define a \emph{multiplicity} of a pair of vertices $(w,v)\in W\times V$ to be the cardinality $m_{w,v}^E$ of the set $\{e\in E^1:r(e)=w, s(e)=v\}=r^{-1}(w)\cap s^{-1}(v)$. We say that $E=(E^1,s,r)$ \emph{has no multiple edges} if every $(w,v)\in W\times V$ has multiplicity at most one. We say that two graphs $E=(E^1,r,s)$ and $F=(F^1,r',s')$ from $V$ to $W$ \emph{are equivalent}, written $E\sim F$, if there is a bijection $\Phi:E^1\to F^1$ such that $s' \circ \Phi =s$ and $r' \circ \Phi =r$. Clearly, two graphs are equivalent if and only if the corresponding multiplicities coincide: $m_{w,v}^E=m_{w,v}^F$ for all $(w,v)\in W\times V$. In other words, up to equivalence, a directed graph is given by a matrix $\{m_{v,w}\}_{v\in V,w\in W }$ of cardinal numbers. Obviously, given any such matrix $M=\{m_{v,w}\}_{v\in V,w\in W}$ there is a graph $E_M=(E^1,r,s)$ whose multiplicities coincide with $M$.

Now, if $f:V\to W$ is a multivalued map then we may associate to it a directed graph $E_{f}:=(\{(w,v)\in W\times V: w\in f(v)\} ,r,s)$ where $r(w,v):=w$ and $s(w,v):=v$. Clearly, $E_{f}$ has no multiple edges and $f(v)=r(s^{-1}(v))$. Conversely, for every directed graph $E=(E^1,r,s)$ from $V$ to $W$ the formula $f^E(v):=r(s^{-1}(v))$, $v\in V$, defines a multivalued map $f^E:V\to W$. We have $E\sim E_{f^E}$ if and only if $E$ has no multiple edges.

In case $V=W$, we will usually write $E^0=V=W$ and $E=(E^0,E^1,s,r)$, and call $E$ a \emph{self-graph}, a \emph{directed graph}, or a \emph{graph with vertex set $E^0$}. In this case, for $n = 2, 3,...$, we define the space $E^n$ of \emph{paths with length }$n$ by
$$
E^n:=\{ (e_n,\dots,e_{2},e_1) \in E^1\times ... E^1\times E^1: r(e_{k})=s(e_{k+1}) \text{ for all } 0 < k < n\}.
$$
Then $(E^0,E^n, r^n,s^n)$ where $r^n(e_n,\dots,e_{2},e_1):=r(e_n)$ and $s^n(e_n,\dots,e_{2},e_1):=s(e_1)$, is a directed graph. We define the space $E^\infty$ of infinite paths and the map $r^\infty:E^\infty\to E^0$ in an obvious way. Let $e=(e_n,\dots,e_{2},e_1) \in E^n$ be a finite path. We call the vertices $s(e_1)$ and $r(e_k)$, $k=1,...,n$, the \emph{base points} of $e$. We say, after \cite[Definition 5.5]{ka1}, that the path $e$ \emph{non-returning} if 	 $e_k\neq e_1$ for every $k=2,...,n$. A path $e=(e_n,\dots,e_{2},e_1) \in E^n$ is a \emph{cycle} if $r(e_n)=s(e_1)$. If $e=(e_n,\dots,e_1) $ is a cycle, we say that $e$ \emph{has no entrances} if $r^{-1}(r(e_k))=\{e_k\}$ for every $k=1,\dots,n$. For any subsets $V, W\subseteq E^0$ we define the restricted graph ${_W}E_{V}$ to be the graph from $V$ to $W$ with the set of edges ${_{W} E}_{V}^1:=\{e\in E^1: s(e)\in V\text{ and } r(e)\in W\}=s^{-1}(V)\cap r^{-1}(W)$, and source and range maps being $s|_V$ and $r|_W$. The following definition is an obvious generalization of \cite[Definition 5.4]{ka1}, cf. \cite[Proposition 6.12]{ka4}.

\begin{defn}\label{top_free_graph}
Let $E=(E^0,E^1,r,s)$ be a directed graph whose set of vertices $E^0$ is a topological space. We say that $E$ is \emph{topologically free} if for every $n>0$ the set of base points of cycles of length $n$ which have no entrances has empty interior. More generally, let $U$ be an open subset of $E^0$. We say that $E$ is \emph{topologically free on $U$} if for every $n>0$ the set of base points of cycles in $({_U}E_U)^n$ which have no entrances in the initial graph $E$ has empty interior.
\end{defn}

In the proof of uniqueness theorem we will use a slightly stronger version of the above condition. We say that the \emph{set $V$ has a past in the restricted graph} ${_U}E_U$ if $r^{-n}(V)\subseteq ({_U}E_U)^n$ for every $n\geq 0$.

\begin{defn}\label{top_free_graph1}
Let $E=(E^0,E^1,r,s)$ be a directed graph where $E^0$ is a topological space. Let $U$ be an open subset of $E^0$. We say that $E$ is \emph{strongly topologically free on $U$} if for every non-empty open set $V\subseteq r^{\infty}(E^\infty)$ that has a past in ${_UE}_U$, and for every $n>0$ there is a non-returning path $e\in r^{-m}(V)$ with $m\geq n$. We say that $E$ is \emph{strongly topologically free} if it is strongly topologically free on $U=E^0$.
\end{defn}

\begin{ex} 
Consider a multivalued map $f:V\to V$ where $V=\{v_0,v_1\}$ is equipped with the topology $\tau=\{\emptyset, V, \{v_0\}\}$, and $f(v_1)=V$ and $f(v_0)=\emptyset$. The corresponding graph $E_f$ is topologically free because $\{v_1\}$ has empty interior, but $E$ is not strongly topologically free because every path of length $n>2$ that ends in $\{v_0\}$ is returning.
\end{ex}

Fortunately, the two notions of topological freeness coincide for continuous graphs:

\begin{defn}
We say that a graph $E$ is \emph{continuous} 	if the set of vertices $E^0$ is a topological space and the associated multivalued map $r\circ s^{-1}:E^0\to E^0$ is continuous, that is for every open set $V\subseteq	 E^0$ the set $E^{-1}(V):=s(r^{-1}(V))$ is open 	in $E^0$. 
\end{defn}

Note that if $E$ is continuous then for every $n>0$ the graph $(E^0,E^n, r^n,s^n)$ is continuous. Topological graphs considered by Katsura \cite{ka1} and the graphs underlying topological quivers \cite{mt} are continuous. 

\begin{lem}\label{lem:topological_freeness} 
Let $E=(E^0,E^1,r,s)$ be a directed graph such that $E^0$ is a topological space. Let $U$ be an open subset of $E^0$. Consider the following conditions:
\begin{itemize}
\item[(i)] $E$ is strongly topologically free on $U$;
\item[(ii)] $E$ is topologically free on $U$. 
\end{itemize}
Then (i)$\Rightarrow$(ii). If ${_UE}_U$ is continuous\footnote{In fact it suffices to assume that ${_UE}_U$ is ``quasi-continuous'' in the sense that for every open $V\subseteq	 U$, if ${_UE}_U^{-1}(V)=s(r^{-1}(V))$ is non-empty, then it has a non-empty interior.}, then (i)$\Leftrightarrow$(ii).
\end{lem}

\begin{proof} 
(i)$\Rightarrow$(ii). Suppose $E$ is not topologically free on $U$. Then there are $n>0$ and a non-empty open set $V$ consisting of base points of cycles in $({_UE}_U)^n$ without entrances in $E$. Clearly, $V\subseteq r^{\infty}(E^{\infty})$ has a past in ${_UE}_U$ and every $e\in r^{-m}(V)$ with $m\geq n+1$ is returning.

(ii)$\Rightarrow$(i). Assume, for contradiction, that there are $n>0$ and a non-empty open set $V\subseteq r^{\infty}(E^\infty)$ such that every path in $E$ that ends in $V$ is a path in ${_UE}_U$, and every path $e\in r^{-m}(V)$ with $m\geq n$ is returning (necessarily in ${_UE}_U$). By the proof of \cite[Lemma 5.9]{ka1} we get that for every $(e_n,\dots,e_1) \in r^{-n}(V)$, there is $k_0$ with $2 \leq k_0\leq n$ such that $(e_{k_0},\dots,e_1)$ is a cycle without entrances in ${_UE}_U$ (and hence also necessarily in $E$). Thus, using continuity of ${_UE}_U$, we see that $s^n(r^{-n}(V))$ is a non-empty open set that consists of base points of cycles in ${_UE}_U$, of length $(n-1)!$, without entrances in $E$.
\end{proof}

Both version of	topological freeness of $E$ on $U$ depend only on the restriction of $E$ to the union of $U$ with $E^{-1}(U)=s(r^{-1}(U))$:

\begin{lem}\label{lem:stupid_lemma}
Let $E=(E^0,E^1,r,s)$ be a directed graph where $E^0$ is a topological space. Let $U\subseteq E^0$ be an open set. The graph $E$ is (strongly) topologically free on $U$ if and only if its restriction to $U \cup E^{-1}(U)$ is (strongly) topologically free on $U$.
\end{lem}

\begin{proof}
For topological freeness it suffices to note that a cycle in $({_U}E_U)^n$ has an entrance in $E$ if and only if it has an entrance in ${_{U \cup E^{-1}(U)}}E_{U \cup E^{-1}(U)}$. For strong topological freeness note that a set $V\subseteq U$ has the past in ${_U}E_U$ treated as a restriction of $E$ if and only if it has the past in ${_U}E_U$ treated as the restriction of ${_{U\cup E^{-1}(U)}}E_{U \cup E^{-1}(U)}$.
\end{proof}

Now we explain	 the relationship between topological freeness for graphs and topological aperiodicity for multivalued maps. In addition we will introduce yet another condition for multivalued maps which will also be useful.

\begin{lem}\label{lem:topological_freeness_vs_aperiodicity}
Let $E=(E^0,E^1,r,s)$ be a directed graph such that $E^0$ is a topological space. Let $U$ be an open subset of $E^0$. Consider the following conditions:
\begin{itemize}
\item[(i)] the restricted multivalued map ${_U}(r\circ s^{-1})_U$ is topologically aperiodic; 
\item[(ii)] for every non-empty open set $V\subseteq r^{\infty}(({_UE}_U)^{\infty})$ and every $n>0$ there is $v\in V$ such that for every path $(e_n,...,e_1)\in r^{-n}(v)\cap ({_UE}_U)^{n}$ we have $v=r(e_n) \neq s(e_k)$ for every $k=1,...,n$;\item[(iii)] for every non-empty open set $V\subseteq r^{\infty}(E^{\infty})$ that has the past in ${_UE}_U$, and every $n>0$ there is a path $(e_n,...,e_1)\in r^{-n}(V)$ such that $r(e_n)\neq s(e_k)$ for every $k=1,...,n$;
\item[(iv)] $E$ is topologically free on $U$.
\end{itemize}
Then (i)$\Leftrightarrow$(ii)$\Rightarrow$(iii)$\Rightarrow$(iv). If $r:{_UE}^1\to U$ is injective, then (iv)$\Rightarrow$(i). If $s:{_UE}^1_U\to U$ is injective, then (iv)$\Rightarrow$(iii).
\end{lem}

\begin{proof}
(ii)$\Rightarrow$(i). Assume that ${_U}(r\circ s^{-1})_U$ is not topologically aperiodic. Then there are a non-empty open set $V\subseteq U$ and $n>0$ such that every $v\in V$ is a starting point of a cycle $e_v$ in ${_UE}_U$ of length not greater than $n$. This implies that $V\subseteq r^{\infty}(({_UE}_U)^{\infty})$ and for each $v\in V$ we may extend the cycle $e_v$ to a path of the form $e_ve \in r^{-n}(v)\cap ({_UE}_U)^{n}$. Since $r(e_v)=s(e_v)=v$, we see that (ii) fails.

(i)$\Rightarrow$(ii). Assume that (ii) fails. Let $V\subseteq r^{\infty}(({_UE}_U)^{\infty})$ be a non-empty open set and let $n>0$ be such that for every $v\in V$ there exists a path $(e_n,...,e_1)\in r^{-n}(v)\cap ({_UE}_U)^{n}$ with $v=r(e_n) =s(e_k)$ for some $k=1,...,n$. This implies that every point in $V$ is a starting point of cycle in $({_UE}_U)^{n!}$. Hence (i) fails.

(ii)$\Rightarrow$(iii). This obvious.

(iii)$\Rightarrow$(iv). Assume that $E$ is not topologically free. Let $n>0$ and let $V\subseteq r^{\infty}(E^\infty)$ be a non-empty open set that has a past in ${_UE}_U$ and every path $e\in r^{-m}(V)$ with $m\geq n$ is returning. By the proof of \cite[Lemma 5.9]{ka1} we get that for every $(e_n,\dots,e_1) \in r^{-n}(V)$, there is $k_0$ with $2 \leq k_0\leq n$ such that $(e_{k_0},\dots,e_1)$ is a cycle without entrances in ${_UE}_U$ (and hence also necessarily in $E$). This implies that for every $(e_{n+1},\dots,e_1) \in r^{-(n+1)}(V)$ we have $r(e_{n+1})\neq s(e_k)$ for some $k=1,...,n$.

Let $V$ be a non-empty open set consisting of base points of cycles in $({_U}E_U)^n$ without entrances in $E$. Then $V\subseteq r^{\infty}(E^{\infty})$ has a past in ${_UE}_U$ and for every path $(e_n,...,e_1)\in r^{-n}(V)$ we have $r(e_n)=s(e_1)$.

(iv)$\Rightarrow$(i). Suppose that $r:{_UE}^1\to U$ is injective. Then every cycle in $E$ that ends in $U$ has no entrances. Thus $E$ is topologically free on $U$ if and only if for every $n>0$ the set of base points of cycles in $({_UE}_U)^{n}$ has empty interior. The latter is clearly equivalent to topological aperiodicity of ${_U}(r\circ s^{-1})_U$. This shows (iv)$\Rightarrow$(i).

(iv)$\Rightarrow$(iii). Suppose that $s:{_U}E^1_U\to U$ is injective and assume that (iii) does not hold. Let $V\subseteq r^{\infty}(E^{\infty})$ be a non-empty open set which has a past in ${_UE}_U$, and let $n>0$ be such that for every path $(e_n,...,e_1)\in r^{-n}(V)$ we have $r(e_n)=s(e_k)$ for some $k=1,...,n$. Thus $V$ consists of base points of cycles (of period $n!$). We claim that these cycles have no entrances. Indeed, let $v \in V$ and let $k=1,...,n$ be the smallest number for which there is a cycle $\mu_v:=(e_n,...,e_k)$ in $E$ with $v=r(e_n)=s(e_k)$ (then $\mu_v$ is necessarily a cycle in ${_UE}_U$). Assume on the contrary that $(e_n,...,e_k)$ has an entrance, that is for some $k'=1,...,k$, there is $e_{k'}'\neq e_{k'}$ in $r^{-1}(r(e_{k'}))$. By our assumption the path $(e_n,...,e_{k-1},e_{k'}')$ can be extended to a cycle $\mu_v':=(e_n,...,e_{k-1},e_{k'}', e_{k'+1}',...,e_l' )$ in ${_UE}_U$, $l'\leq n$. However, since $s:{_UE}^1_U\to U$ is injective, every path in ${_UE}_U$ is determined by its length and starting point. We see that either $\mu_v=\mu_v'$ or $\mu_v'$ is a concatenation of a number of cycles $\mu_v$. In both cases we get $e_{k'}'=e_{k'}$, a contradiction.
\end{proof}

\begin{defn}\label{def:weakly_top_aperiod} 
Let $f:W\to W$ be a multivalued map on a topological space $W$. Let $U$ be an open subset of $W$. We say that $f$ is \emph{weakly topologically aperiodic on $U$} if the graph $E_f$ associated to $f$ satisfies condition (iii) in Lemma~\ref{lem:topological_freeness_vs_aperiodicity}. We say that $f$ is \emph{topologically aperiodic on $U$} if ${_Uf}_U:U\to U$ is topologically aperiodic.
\end{defn}

The following example illustrates the relationships between weak topological aperiodicity, topological aperiodicity, and topological freeness.

\begin{ex}\label{partial map example} Let $\varphi:\Delta \to V$ be an ordinary map defined on a subset $\Delta\subseteq V$ of a topological space $V$. Treating $\varphi^{-1}$ as a multivalued map the corresponding graph $E_{\varphi^{-1}}$ is given by putting $E^1_{\varphi^{-1}}:=\{(x,y):x\in \varphi^{-1}(y)\}$, $r(x,y):=x$ and $s(x,y):=y$. Note that $r:E^1_{\varphi^{-1}}\to \Delta$ is injective. In view of Lemma~\ref{lem:topological_freeness_vs_aperiodicity}, for any open $U\subseteq V$, we get 
\begin{align*}
\varphi\text{ is topologically aperiodic on $U$} &\,\,\Longleftrightarrow\,\,\varphi^{-1}\text{ is topologically aperiodic on $U$}\\
&\,\,\Longleftrightarrow\,\,\text{the graph $E_{\varphi^{-1}}$ is topologically free on $U$}.
\end{align*}
Treating $\varphi$ as a multivalued map the associated graph $E_\varphi$ is given by putting $E^1_\varphi:=\{(\varphi(x),x):x\in \Delta\}$, $r(y,x):=y$ and $s(y,x):=x$. Note that the source map $s$ is injective. Recall after \cite[Definition 4.8]{kwa-rever}, see also \cite[Definition 2.37]{kwa-endo}, that $\varphi$ is \emph{topologically free outside a set $Y\subseteq V$} if the set of periodic points whose orbits do not intersect $Y$ and have no entrances have empty interior (a periodic orbit $\OO=\{x, \varphi(x),..., \varphi^{n-1}(x)\}$ of a periodic point $x=\varphi^{n}(x)$ \emph{has an entrance} if there is $y\in\Delta\setminus \OO$ such that $\varphi(y)\in \OO$). In view of Lemma~\ref{lem:topological_freeness_vs_aperiodicity}, for any open $U\subseteq V$, we get 
\begin{align*}
\varphi\text{ is weakly topologically aperiodic on $U$} &\,\,\Longleftrightarrow\,\, \text{the graph $E_{\varphi}$ is topologically free on $U$} \\
&\,\,\Longleftrightarrow\,\,\varphi \text{ is topologically free outside $V\setminus U$}.
\end{align*}
Let us restrict to the case where $V=\{v_0,v_1\}$ is a discrete space. If $\Delta=V$ and $\varphi(V)=v_1$, then $\varphi$ is weakly topologically aperiodic but is not topologically aperiodic on $V$. If $f:V\to V$ is a multivalued map such that $f(v_1)=V$ and $f(v_0)=\{v_1\}$, then $f$ is not weakly topologically aperiodic on $V$ but the corresponding graph $E_f$ is topologically free.
\end{ex}

\section{Multivalued maps dual to \(C^*\)-correspondences} \label{sect:Multivalued maps}

Let $\alpha:A\rightarrow B$ be a $*$-homomorphism between two $C^*$-algebras $A$ and $B$. Following \cite[Definition 4.1]{kwa-szym} we define the \emph{multivalued map $\sal:\SB\rightarrow \SA$ dual to $\alpha$} by the formula
$$
\sal([\rho]):= \{[\pi]\in \SA: \pi\leq \rho\circ \al\}, \qquad [\rho]\in \SB.
$$

\begin{lem}[Proposition 4.2 in \cite{kwa-szym}]\label{range of duals to homomorphisms} 
Let $\sal:\SB\rightarrow \SA$ be the multivalued map dual to a $*$-homomorphism $\al:A\to B$. Then
$$
\sal(\SB)=\SA\setminus \widehat{\ker{\alpha}} \quad \text{ and }\quad \sal^{-1}(\SA)\subseteq \widehat{B\alpha(A)B}.
$$
In particular, $\sal$ is onto $\widehat{A}$ if and only if $\alpha$ is injective. Moreover, if $B$ is liminal, then $\sal^{-1}(\SA) = \widehat{B\alpha(A)B}$ and $\sal$ is continuous.
\end{lem}

The following definition is a generalization of \cite[Definition 4.4]{kwa-szym}, originally formulated for regular $C^*$-correspondences.

\begin{defn}
Let $A$ and $B$ be $C^*$-algebras, and let $X:A\to B$ be a $C^*$-correspondence from $A$ to $B$ with left action $\phi_X:A\rightarrow \LL(X)$. We define the \emph{multivalued map $\SX:\SB\to \SA$ dual to $X$} as the composition of multivalued maps $\widehat{\phi_X}\circ [X\dashind]$, cf. Definition~\ref{def:dual_partial_homeomorphism}.
\end{defn}

It follows from the definition that the multivalued map $\SX$ is given by the formula
$$
\SX([\rho]):= \{[\pi]\in \SA: \pi \leq X\dashind(\rho)\circ \phi_X\}, \qquad [\rho]\in \SB.
$$

\begin{lem}\label{lem:restrictions of graphs vs correspondences} 
Let $A$ and $B$ be $C^*$-algebras. Let $X:A\to B$ be a $C^*$-correspondence, and let $\SX:\SB\to\SA$ be the associated dual multivalued map. Given an ideal $I$ in $A$ and an ideal $J$ in $B$, the following statements hold:
\begin{itemize}
\item[(i)] $IX$ is a $C^*$-correspondence from $I$ to $B$ whose dual multivalued map coincides with the restriction ${_{\widehat{I}}}\X$.
\item[(ii)] $XJ$	is a $C^*$-correspondence from $A$ to $J$ whose dual multivalued map coincides with the restriction $\X{_{\widehat{J}}}$.
\item[(iii)] $IXJ$ is a $C^*$-correspondence from $I$ to $J$ whose multivalued map coincides with the restriction ${_{\widehat{I}}}\X{_{\widehat{J}}}$.
\end{itemize}
\end{lem}

\begin{proof} 
(i). Let $\rho\in \Irr(B)$ and $\pi\in \Irr(A)$ be such that $\pi \leq X\dashind(\rho)\circ \phi_X$. If $\pi(I)\neq 0$, then by the irreducibility of $\pi(I)$, we have $H_\pi=\phi_X(I)H_\pi\subseteq IX\otimes_\rho H_\rho$. Thus, $[\pi]\in \widehat{IX}([\rho])$ if and only if $[\pi]\in {_{\widehat{I}}\X}([\rho])$.

(ii). Let $\pi\in \Irr(A)$ and $\rho\in \Irr(J)$ be such that $\pi \leq X\dashind(\rho)\circ \phi_X$. Then $X\otimes_\rho H_\rho=X\otimes_\rho \rho(J)H_\pi=XJ\otimes_\rho H_\rho$, which gives the assertion.

(iii). It follows from (i) and (ii).
\end{proof}

For the sake of simplicity, in the remainder of this section we only consider \(C^*\)-correspondences over \(A\) (though the following facts could be easily extended to $C^*$-correspondences from \(A\) to \(B\)).

The following lemma is a generalization of \cite[Proposition 4.5]{kwa-szym}.

\begin{lem}\label{range of duals to correspondences}
Let $X$ be a $C^*$-correspondence over $A$ and let $\SX:\SA\to \SA$ be the multivalued map dual to $X$.
\begin{itemize}
\item[(i)] We have $\widehat{J(X)}\setminus \widehat{\ker\phi_X} \subseteq \X(\SA)\subseteq \widehat{A}\setminus \widehat{\ker\phi_X}$. In particular, the range of $\SX$ contains $\widehat{J_X}$.
\item[(ii)] For any ideal $J$ in $A$ we have $\SX^{-1}(\widehat{J})\subseteq \widehat{X(J)}$. Moreover, if $X(J)$ is liminal and $J\subseteq J(X)$, then $\SX^{-1}(\widehat{J})=\widehat{X(J)}$.
\end{itemize}
\end{lem}

\begin{proof}
(i). By Lemma~\ref{range of duals to homomorphisms}, the range of $\widehat{\phi_X}$ is $\SA\setminus \widehat{\ker\phi_X}$, and hence the range of $\SX=\widehat{\phi_X}\circ [X\dashind]$ is contained in $\SA\setminus \widehat{\ker\phi_X}$ too. Now let $\pi\in \Irr(J(X))$ such that $\pi(\ker\phi_X)= 0$, that is $[\pi]\in \widehat{J(X)}\setminus \widehat{\ker\phi_X}$. By \cite[Proposition 3.2.1]{Dixmier}, $\pi$ factors to an irreducible representation $\tilde{\pi}\in \Irr(J(X)/\ker\phi_X)$, and since $\phi_X$ factors to an isomorphism $\tilde{\phi_X}:J(X)/\ker\phi_X\rightarrow \phi_X(J(X))\subseteq \K(X)$, we get that $\tilde{\pi}\circ \tilde{\phi_X}^{-1}$ is an irreducible representation of $\phi_X(J(X))$. Extending $\tilde{\pi}\circ \tilde{\phi_X}^{-1}$ to an irreducible representation $\bar{\pi}$ of $\K(X)$, so that $\tilde{\pi}\circ \tilde{\phi_X}^{-1}\leq \bar{\pi}$, and applying $\tilde{\phi_X}$, we have $\tilde{\pi}\leq \bar{\pi}\circ \tilde{\phi_X}$. But then $\pi\leq \bar{\pi}\circ\phi_X$. Thus, if we define $\rho:=X^*\dashind(\bar{\pi})$ we get $[\pi]\in \SX([\rho])$. Hence $[\pi]\in \SX(\SA)$.

(ii). Let $\pi,\rho\in\Irr(A)$ with $\pi\leq X\dashind(\rho)\circ \phi_X$, that is $[\pi] \in \SX([\rho])$, and suppose that $[\pi]\in \widehat{J}$. So $\pi(J)\neq 0$. This implies that $0\neq \pi(J)H_\pi\subseteq \phi_X(J)X\otimes_\rho H_\rho$, and hence $0\neq\langle \phi_X(J) X\otimes_{\rho} H_{\rho} , \phi_X(J) X\otimes_{\rho} H_{\rho} \rangle= \langle H_{\rho} , \rho(\langle \phi_X(J) X , \phi_X(J) X\rangle_A) H_{\rho}\rangle=\langle H_\rho, \rho(X(J))H_\rho\rangle\,.$ Therefore, $\rho( X(J)) \neq 0$, and this shows that $[\rho]\in \widehat{ X(J)}$.

Now assume that $ X(J)$ is liminal and $J\subseteq J(X)$. Let $\rho\in\Irr(A)$ be such that $\rho(X(J))\neq 0$, that is $[\rho]\in\widehat{X(J)}$. Note that $JX$ establishes Morita equivalence between $\K(JX)$ and $X(J)=\langle JX,JX\rangle_A$. Hence $\K(JX)$ is liminal, cf. \cite{HRW}, and therefore $JX\dashind(\rho)(\K(JX))=\K(JX\otimes_\rho H_\rho)$. The inclusion $\phi_X(J)\subseteq \K(X)$ implies that in fact $\phi_X(J)\subseteq \K(JX)$. Since $X\dashind(\rho)(\K(JX))|_{JX\otimes_\rho H_\rho}=JX\dashind(\rho)(\K(JX))$ and $X\dashind(\rho)\circ\phi_X(J)$ vanishes on the orthogonal complement of $JX\otimes_\rho H_\rho$, we see that $X\dashind(\rho)\circ\phi_X(J)$ consists of compact operators. Accordingly, by \cite[5.4.13]{Dixmier}, the representation $X\dashind(\rho)\circ\phi_X$ contains an irreducible subrepresentation $\pi$ of $A$ with $\pi(J)\neq 0$, that is $[\pi]\in \widehat{J}$. Thus, $[\pi]\in \SX([\rho])$, and hence $[\rho]\in \SX^{-1}(\widehat{J})$, as desired.
\end{proof}

\begin{rem}
In general, the range of $\SX$ is not equal to $\widehat{J(X)}\setminus \widehat{\ker\phi_X}$. Consider for instance an infinite dimensional space $H$ as a $C^*$-correspondence over $\C$. Then $\widehat{J(X)}=\emptyset\neq [1_{\C}]=r(\SX^1)$.
\end{rem}

\begin{cor}\label{continuity of dual graphs}
Let $X$ be a $C^*$-correspondence over $A$ and let $J$ be an ideal in $J_X$. Consider the multivalued map $\widehat{Y}$ dual to the $C^*$-correspondence $Y:=JXJ$ over the $C^*$-algebra $J$. If $X(J)$ is liminal, then $\widehat{Y}$ is continuous.
\end{cor} 

\begin{proof}
Let $I$ be an ideal in $J$. Since $X(I)\subseteq X(J)$ is liminal, we have $\SX^{-1}(\widehat{I})= \widehat{X(I)}$ by Lemma~\ref{range of duals to correspondences}(ii). Hence in view of Lemma~\ref{lem:restrictions of graphs vs correspondences} we get $\widehat{Y}^{-1}(\widehat{I})=\SX^{-1}(\widehat{I})\cap \widehat{J}= \widehat{X(I)}\cap \widehat{J}=\widehat{Y(I)}$. 
\end{proof}

\begin{cor}\label{invariance via dual graphs}
Let $X$ be a $C^*$-correspondence over $A$ and let $J$ be an ideal of $A$. If $J$ is positively $X$-invariant, then $\widehat{J}$ is negatively invariant with respect to $\SX$. If $X(J)$ is liminal and $J\subseteq J(X)$, then $J$ is a positively $X$-invariant if and only if $\widehat{J}$ is negatively $\SX$-invariant.
\end{cor} 

\begin{proof}
If $J$ is a positively invariant ideal of $A$, then $\widehat{X(J)}\subseteq \widehat{J}$. Hence $\SX^{-1}(J)\subseteq \widehat{X(J)}\subseteq \widehat{J}$ by Lemma~\ref{range of duals to correspondences}(ii). If we assume that $X(J)$ is liminal, $J\subseteq J(X)$ and $\widehat{J}$ is negatively invariant with respect to $\SX$, then by Lemma~\ref{range of duals to correspondences}(ii) we have $\widehat{X(J)}=\SX^{-1}(\widehat{J})\subseteq \widehat{J}$. Therefore, $X(J)$ is contained in $J$.
\end{proof}

The following proposition is a generalization of \cite[Proposition 4.6]{kwa-szym}. 

\begin{prop}\label{tensoring vs duals}
Let $X$ and $Y$ be two $C^*$-correspondences over a $C^*$-algebra $A$, and let $\otimes 1_X:\LL(Y)\to\LL(Y\otimes X)$ be the $*$-homomorphism defined by $T\to T\otimes 1_X$ for $T\in \LL(Y)$. Then the following diagram of multivalued maps commute:
\begin{equation}\label{correspondences vs compacts graphs}
\xymatrix{ \SA \ar[rr]^{[(Y\otimes X)\dashind]\,\,\,\,\,} \ar[d]^{\X} &\qquad & \widehat{\LL(Y\otimes X)} \ar[d]^{\widehat{\otimes 1_X}} \\
\SA \ar[rr]^{[Y\dashind]} & & \widehat{\LL(Y)} }.
\end{equation} 
Moreover, if $J$ is an ideal of $J_X$ then $\SX^{-1}(\widehat{\langle YJ,YJ\rangle}_A)\subseteq \widehat{\langle YJ\otimes X,YJ\otimes X\rangle}_A$, and the diagram \eqref{correspondences vs compacts graphs} restricts to the following commutative diagram:
\begin{equation}\label{correspondences vs compacts graphs2}
\xymatrix{ \widehat{\langle Y \otimes X, Y \otimes X\rangle}_A \ar[rr]^{[(Y\otimes X)\dashind]\,\,\,\,\,} \ar[d]^{\X} &\qquad & \widehat{\K(Y\otimes X)} \ar[d]^{\widehat{\otimes 1_X}} \\
\widehat{\langle YJ, YJ\rangle}_A \ar[rr]^{[Y\dashind]} & & \widehat{\K(Y J)} },
\end{equation} 
where the vertical multivalued maps are surjective, and the horizontal arrows are homeomorphisms.
\end{prop}

\begin{proof}
First we describe the multivalued maps $E:=[Y\dashind]\circ \SX$ and $F:=\widehat{\otimes 1_X}\circ [(Y\otimes X)\dashind]$ acting from $\SA$ to (subsets of) $\widehat{\LL(Y)}$. For each $\rho\in \Irr(A)$ we have
$$E([\rho]) =\{[Y\dashind(\pi)] : \pi\in \Irr(A),\,\pi\leq X\dashind(\rho)\circ\phi_X \text{ and }\pi(\langle Y,Y\rangle_A)\neq 0 \},$$
and
$$F([\rho])=\{[\sigma]:\sigma\in\Irr(\LL(Y)),\,\sigma\leq (Y\otimes X)\dashind(\rho)\circ 1_X\}.$$
We need to show that $E([\rho])=F([\rho])$. To this end, let $[Y\dashind(\pi)] \in E([\rho])$. Then 
$$Y\dashind(\pi)\leq Y\dashind(X\dashind(\rho)\circ\phi_X).$$ 
It is straightforward to check that $Y\dashind(X\dashind(\rho)\circ\phi_X)=(Y\otimes X)\dashind(\rho)\circ 1_X$ as representation of $\LL(Y)$. Thus, $Y\dashind(\pi)\leq (Y\otimes X)\dashind(\rho)\circ 1_X$, and hence $[Y\dashind(\pi)] \in F([\rho])$.

Conversely, let $[\sigma] \in F([\rho])$. Since $\sigma \le (Y\otimes X)\dashind(\rho) \circ (\otimes 1_X)$, we see that $H_\sigma$ is a closed subspace of $Y\otimes X\otimes H_{\rho}$ such that $(T\otimes 1_X\otimes 1_{H_{\rho}})(H_\sigma)\subseteq H_\sigma$ for every $T\in\LL(Y)$. Since $\K(Y)\subseteq \LL(Y)$ acts transitively on $Y$, it follows that 
\begin{align*}
H_\pi:&=\{h\in X\otimes H_{\rho}:y\otimes h\in H_\sigma\text{ for some }y\in Y\}\\&=\{h\in X\otimes H_{\rho}:y\otimes h\in H_\sigma\text{ for all }y\in Y\}
\end{align*}
is a closed subspace of $X\otimes H_{\rho}$ such that $Y\otimes H_\pi=H_\sigma$ and $\phi_X(\langle Y,Y\rangle_A)H_\pi=H_\pi$. Define $\pi:A\to\B(H_\pi)$ by $\pi(a)h=X\dashind(\rho)(\phi_X(a))h$. Then $\pi$ is a representation of $A$ and $\pi\le X\dashind (\rho)\circ \phi_X$. Let $T\in\LL(Y)$, $y\in Y$, and $h\in H_\pi$. Since $\sigma \le (Y\otimes X)\dashind(\rho) \circ (\otimes 1_X)$ and $H_\pi\subseteq X\otimes H_{\rho}$, we have
$$Y\dashind(\pi)(T)(y\otimes h)=Ty\otimes h=\sigma(T)(y\otimes h).$$
Because $Y\otimes H_\pi=H_\sigma$, this shows that $Y\dashind(\pi)=\sigma$. Since $Y\dashind(\pi)=\sigma$, $Y\dashind(\pi)$ is irreducible. Since $Y$ is a $\K(Y)-\langle Y,Y\rangle_A$-imprimitivity bimodule, $\pi$ is irreducible on $\langle Y,Y\rangle_A$. So $\pi$ is irreducible. Hence $[\sigma]=[Y\dashind(\pi)]\in E([\rho])$. Thus the diagram \eqref{correspondences vs compacts graphs} commutes.

Now, let $J$ be an ideal in $J_X$. The map $\otimes 1_X:\K(Y J) \to \K(Y\otimes X)$ is injective by Lemma~\ref{lemma on tensoring regular correspondences}(ii). Hence the dual multivalued map $\widehat{\K(Y\otimes X)} \stackrel{\widehat{\otimes 1_X}}{\rightarrow} \widehat{\K(Y J)}$ is surjective, by Lemma~\ref{range of duals to homomorphisms}. Since $\widehat{J}\subseteq\widehat{J_X}\subseteq\widehat{J(X)}\setminus\widehat{\ker\phi_X}$, it follows from Lemma~\ref{range of duals to correspondences} (i) that the multivalued map $\X:\SA{\longrightarrow}\widehat{J} $ is surjective. According to Lemma~\ref{range of duals to correspondences} (ii) the domain of $\X:\SA\longrightarrow \widehat{\langle YJ, YJ\rangle}_A $ is contained in the spectrum of $X(Y(J))=\langle X , \phi_X(\langle YJ, YJ\rangle_A) X\rangle_A=\langle YJ \otimes X, YJ \otimes X\rangle_A$. Since $\langle YJ, YJ\rangle_A\subseteq J$, the surjective multivalued map $\X:\SA{\longrightarrow}\widehat{J} $ restricts to a surjective multivalued map $\X:\widehat{\langle Y \otimes X, Y \otimes X\rangle}_A\longrightarrow \widehat{\langle YJ, YJ\rangle}_A $. Since $\widehat{\langle Y \otimes X, Y \otimes X\rangle}_A\stackrel{[(Y \otimes X)\dashind]}{\longrightarrow} \widehat{\K(Y \otimes X)}$ and $\widehat{\langle Y J, Y J\rangle}_A\stackrel{[Y\dashind]}{\longrightarrow} \widehat{\K(Y J)}$ are homeomorphisms, we see that \eqref{correspondences vs compacts graphs} restricts to the commutative diagram \eqref{correspondences vs compacts graphs2}, with arrows having the desired properties. 
\end{proof}

\section{Dual graphs to \(C^*\)-correspondences}\label{sec:Dual graphs}

Let $X$ be a $C^*$-correspondence from $A$ to $B$. We will now adjust the construction of the dual multivalued map $\X$, by counting the multiplicities of the corresponding subrepresentations. In order to do that we will assume that $A$ is of Type I, as for such $C^*$-algebras a convenient multiplicity theory is available.

We will construct the graph $E_X$ dual $X$ by specifying its multiplicities on the set of vertices $\SA$. Thus $E_X$ will be determined only up to equivalence. Let $\rho\in\Irr(A)$. Consider the representation $\sigma:A\to \B(H_\sigma)$ where 	$\sigma:=X\dashind(\rho)\circ\phi_X$ and $H_\sigma:=X\otimes_\rho H_\rho$. Since $A$ is of type $I$, cf. e.g. \cite[Proposition 5.4.9]{Dixmier}, there is a unique, up to permutation, family of mutually orthogonal projections $\{E_j\}_{j\in J}\subseteq \B(H_\sigma)$ such that 
\begin{itemize}
\item[a)] projections $E_j$ belong to the center of $\sigma(A)'$ and sum up to identity on $H_\sigma$;
\item[b)] the subrepresentations $\sigma_j$ of $\sigma$ corresponding to the $E_j$ are mutually disjoint;
\item[c)] $\sigma_j$ is of multiplicity $m_j$;
\item[d)] the cardinals $m_j$'s are mutually distinct.
\end{itemize}
Now, let $[\pi]\in \SA$. If there are no subrepresentation of $\sigma$ equivalent to $\pi$, we put $m_{[\pi],[\rho]}^X:=0$. If there is an irreducible subrepresentation $\pi'\cong\pi$ of $\sigma$, then there is a unique $j\in J$ such that $\pi'\leq \sigma_j$. Since $\sigma_j$ has multiplicity $m_j$, it follows, cf. \cite[Proposition 5.4.7]{Dixmier}, that there are $m_j$ mutually orthogonal representations $\pi^i$ such that $[\pi^i]=[\pi]$, $\oplus_{i}\pi^i\leq \sigma$ and there are no subrepresentations of $\sigma \ominus (\oplus_{i}\pi^i)$ equivalent to $\pi$. In this case we put $m_{[\pi],[\rho]}^X:=m_j$. Note that the above description of cardinals $m_{[\pi],[\rho]}^X$ does not depend on the choice of a representative of $[\rho]$, and thus they are well defined.

\begin{defn}\label{def:dual_graph}
Let $X$ be a $C^*$-correspondence over a $C^*$-algebra $A$ of Type I to a $C^*$-algebra $B$. By a graph dual to $X$ we mean a graph $E_X$ from $\SB$ to $\SA$ with multiplicities $\{m_{[\pi],[\rho]}^X\}_{[\pi]\in \SA,[\rho]\in \SB}$ defined as above.
\end{defn}

\begin{rem}
If $E_X=(E^1_X, r,s)$ is a graph dual to a $C^*$-correspondence $X$ from $A$ to $B$, then the associated multivalued map coincides with the multivalued map dual to $X$ from $\SB$ to $\SA$. That is, we have $\SX([\pi])=r(s^{-1}([\pi]))$ for every $[\pi]\in \SB$. Moreover, given ideals $I$ in $A$ and $J$ in $B$, similarly as in Lemma~\ref{lem:restrictions of graphs vs correspondences}, we see that the graph dual to the restricted $C^*$-correspondence $IXJ$ from $I$ to $J$ can be identified with the restriction of $E_X$ to the graph from $\widehat{J}$ to $\widehat{I}$, i.e. $E_{IXJ}={_{\widehat{I}}}(E_X)_{\widehat{J}}$.
\end{rem}

In the discussion preceding Definition~\ref{def:dual_graph} the cardinals $m_{[\pi],[\rho]}^X$ are uniquely determined by the equivalence classes $[\pi]$ and $[\rho]$. However, the decomposition $\oplus_{i}\pi^i$ of a subrepresentation of $\sigma$ into the $m_{[\pi],[\rho]}^X$-number of copies of $\pi$ is not unique even for the fixed representatives. In order to have a control over the choice of the corresponding decompositions we will use the map $Q$ in the following proposition. This map takes values in the von Neumann enveloping $C^*$-algebra $\K(X)''$ of $\K(X)$.

\begin{lem}\label{lem:stanislaw}
Let $X$ be a $C^*$-correspondence over a $C^*$-algebra $A$. If $J$ is an ideal in $J_X$ such that $X(J)$ is of Type I (resp. liminal), then $J$ is of Type I (resp. liminal). 
\end{lem}

\begin{proof}
Since $JX$ is an $\K(JX)$-$X(J)$-equivalence bimodule, we get that $\K(JX)$ is of Type I (resp. liminal). Thus $J$ is of Type I (resp. liminal) because $\phi_X:J \to \K(JX)$ is an injective homomorphism and being of Type I (resp. liminal) passes to subalgebras. 
\end{proof}

\begin{prop}\label{space_decomposition_according_to_graph}
Let $X$ be a $C^*$-correspondence over a $C^*$-algebra $A$. Let $J$ be an ideal in $J_X$. Suppose that $X(J)$ is of Type I, so that the graph $E_{Y}=(\widehat{J}, E^1, r,s)$ dual to $Y:=JXJ$ is well defined by Lemma~\ref{lem:stanislaw}. There is a map $Q:E^1_Y\to \K(X)''$ such that for any $\rho\in \Irr(A)$ with $\rho(J)\neq 0$, the family $\{X\dashind (\rho)''(Q_e)\}_{e \in s^{-1}([\rho])}\subseteq \B(X \otimes_\rho H_\rho)$ consists of mutually orthogonal projections and 
\begin{equation}\label{decomposition into subspaces}
X\otimes_\rho H_\rho= H_0\oplus \bigoplus_{e \in s^{-1}([\rho])} X\dashind (\rho)''(Q_e) (X \otimes_\rho H_\rho)
\end{equation}
where the space $X\dashind (\rho)''(Q_e) (X \otimes_\rho H_\rho)$, for $e \in s^{-1}([\rho])$, is irreducible for $X\dashind(\rho)\circ\phi_X(J)$ and the equivalence class of the corresponding representation of $J$ is $[r(e)]$, and $H_0$ does not contain non-zero irreducible subspace for $X\dashind(\rho)\circ\phi_X(J)$.
\end{prop}

\begin{proof} 
Let $\rho\in \Irr(A)$ with $\rho(J)\neq 0$. Note that $X\otimes_\rho H_\rho=X\otimes_\rho \rho(J)H_\rho=XJ\otimes_\rho H_\rho$, and hence $JX\otimes_\rho H_\rho=Y\otimes_\rho H_\rho$. As $JX$ is an $\K(JX)$-$X(J)$-equivalence bimodule it follows that $X\dashind (\rho)(\K(JX))$ acts on $JX\otimes_\rho H_\rho=Y\otimes_\rho H_\rho$ in an irreducible way. In particular, $Y\otimes_\rho H_\rho$ is non-zero if and only if $\rho(X(J))\neq 0$. If $\rho(X(J))= 0$, then $[\rho]$ is a source in $E_Y$, i.e. $s^{-1}([\rho])=\emptyset$, and we get \eqref{decomposition into subspaces} by interpreting the empty sum as zero.

Assume then that $\rho(X(J))\neq 0$. The algebra $\K(JX)$ is of Type I because $X(J)$ is of Type I. Hence $X\dashind (\rho)(\K(JX))\supseteq \K(Y\otimes_\rho H_\rho)$. This implies that $X\dashind (\rho)''(\K(JX)'')=\B(Y\otimes_\rho H_\rho)$. Since $\phi_X(J)\subseteq \K(JX)$ we conclude that $X\dashind(\rho)\circ\phi_X(J)$ acts as zero on the orthogonal complement of $Y\otimes_\rho H_\rho$ and $\sigma:=X\dashind(\rho)\circ\phi_X|_J: J \to \B(Y\otimes_\rho H_\rho)$ is a non-degenerate representation. Decomposing $\sigma$ into irreducible parts we get a family of mutually orthogonal projections $\{P_e\}_{e \in s^{-1}([\rho])}\subseteq \sigma(J)'\subseteq \B(Y\otimes_\rho H_\rho)$ such that $\sigma$ compressed to $P_e (Y\otimes_\rho H_\rho)$ is an irreducible representation whose equivalence class is $r(e)\in \widehat{J}$. Moreover, the orthogonal complement $H_0$ of $\oplus_{e \in s^{-1}([\rho])} P_{e}(Y\otimes_\rho H_\rho)$ in $X\otimes_\rho H_\rho$ has no irreducible subrepresentations for $X\dashind(\rho)\circ\phi_X(J)$.

For $e \in s^{-1}([\rho])$ we may find a positive element $Q_e\in\K(JX)'']\subseteq \K(X)''$ such that $X\dashind (\rho)''(Q_e)=P_e$. In this way we get operators satisfying \eqref{decomposition into subspaces}. In particular, if $\rho'\cong\rho$ with the equivalence given by a unitary $U$, then $X\dashind (\rho)\cong X\dashind (\rho')$ with the equivalence given by the unitary $X\dashind(U):X\otimes_{\rho} H_\rho\to X\otimes_{\rho'} H_{\rho'}$ where $x\otimes\xi\mapsto x\otimes U\xi$ for $x\in X$ and $\xi\in H_\rho$. Thus we get a similar decomposition to \eqref{decomposition into subspaces} with $\rho$ replaced by $\rho'$. In other words, the map $s^{-1}([\rho])\ni e \mapsto Q_e\in \K(X)''$ is well defined. Gluing together the maps corresponding to each $[\rho]\in \widehat{J}\subseteq \SA$ we get the desired map defined on $E^1_Y=\bigsqcup_{[\rho]\in \widehat{J}}s^{-1}([\rho])$. 
\end{proof}

We recall, see \cite{Schweizer2}, that if X is a Hilbert $A$-$B$-bimodule, then its Banach space bidual $X''$ is a Hilbert $W^*$-module over the enveloping von Neumann algebras $A''$, $B''$. In particular, the bidual of a Hilbert $A$-module $X$ is a Hilbert $W^*$-bimodule from $\K(X)''$ to $A''$. Moreover, every representation $(\psi_0,\psi_1)$ of the Hilbert $A$-module $X$ on a Hilbert space $H$ extends to a weakly continuous representation $(\psi_0'',\psi_1'')$ of $X''$ on $H$ and the associated representation of $\K(X)''$ coincides with the weakly continuous extension $\psi^{(1)}{''}$ of the representation $\psi^{(1)}$ of $\K(X)$.

Property (iv) in the following lemma will be crucial in the proof of our uniqueness theorem.

\begin{lem}\label{lem:proposition_in_terms_of_representations}
Retain the assumption of Proposition~\ref{space_decomposition_according_to_graph} and let $Q:E^1_Y\to \K(X)''$ be the map defined there ($Y:=JXJ$). Let $\psi=(\psi_0,\psi_1)$ be a representation of the $C^*$-correspondence $X$ on a Hilbert space $H$ and let $\psi''=(\psi_0'',\psi_1'')$ be the corresponding representation of the dual Hilbert module $X''$ over $A''$. 
\begin{itemize}
\item[(i)] If $\pi\leq \psi_0$ is an irreducible subrepresentation of $A$ with $\pi(J)\neq 0$ ($[\pi]\in \widehat{J}$), then the operators $\{\psi^{(1)}{''}(Q_e)\}_{e \in s^{-1}([\pi])}$ restricted to the space $\psi_1(X)H_\pi$ are mutually orthogonal projections. Moreover,
$$
\psi_1(X)H_\pi=K\oplus \bigoplus_{e \in s^{-1}([\pi])} \psi_1''(Q_e X) H_\pi
$$
where for each $e \in s^{-1}([\pi])$ the space $\psi_1''(Q_e X) H_\pi=\psi^{(1)}{''}(Q_e)\psi_1(X)H_\pi$ is irreducible for $\psi_0(A)$ where the corresponding representation of $A$ is equivalent to $[r(e)]\in \widehat{J}$, and $K$ does not contain any non-zero irreducible subspaces for $\psi_0(A)$.

\item[(ii)] If, for $i=1,2$, we have $\pi_i\leq \psi_0$ with $[\pi_i]\in\widehat{J}$ then $ \psi_1''(Q_eX)H_{\pi_1}\,\, \bot \,\, \psi_1''(Q_fX)H_{\pi_2}$ for every $e \in s^{-1}([\pi_1]), \,\, f\in s^{-1}([\pi_2])$ with $e\neq f$.

\item[(iii)] For every $n>0$ and $\pi\leq \psi_0$ with $[\pi]\in\widehat{J}$ we have
$$
\psi_n(X^{\otimes n})H_\pi= \bigoplus_{(e_n,...,e_1)\in s^{-n}([\pi])} \psi_1''(Q_{e_n} X)...\psi_1''(Q_{e_1} X) H_\pi \oplus \bigoplus_{i=0}^{n-1} \psi_i(X^{\otimes i})K_i^n
$$
where $K^n_i$ is a $\psi_0(A)$-invariant subspace which has no non-zero irreducible subspaces for $\psi_0(J)$.

\item[(iv)] Suppose that $\pi\leq \psi_0$ with $[\pi]\in\widehat{J}$ and $(e_n,...,e_1)\in s^{-n}([\pi])$ is a non-returning path, i.e. $e_k\neq e_1$ for $k=2,...,n$. For every $k=1,..., n-1$ we have
$$
\psi_k(X^{\otimes k})\psi_1''(Q_{e_n} X)...\psi_1''(Q_{e_1} X) H_\pi \,\, \bot\,\, \psi_1''(Q_{e_n} X)...\psi_1''(Q_{e_1} X) H_\pi.
$$
\end{itemize}
\end{lem}

\begin{proof}
(i). By Lemma~\ref{representation structure lemma2} we have a unitary $U:\psi_1(X)H\to X\otimes_{\psi_0} H$ where $\psi_1(x)h\mapsto x\otimes_{\psi_0}h$. It intertwines the representations $\overline{\psi^{(1)}}:\LL(X)\to \B(\psi_1(X)H)$ and $X\dashind (\psi_0):\LL(X)\to \B(X\otimes_{\psi_0} H)$. Hence it also intertwines the corresponding extended representations $\overline{\psi^{(1)}}{''}$ and $X\dashind (\psi_0)''$ of $\LL(X)''$. Moreover, identifying $X\otimes_{\pi} H_{\pi}$ with a subspace of $X\otimes_{\psi_0} H$ we have $U(\psi_1(X)H_\pi)=X\otimes_{\pi} H_{\pi}$ and therefore $U$ also intertwines $\psi^{(1)}{''}:\K(X)''\to \B(\psi_1(X)H_\pi)$ and $X\dashind (\psi_0)'':\K(X)''\to \B(X\otimes_{\psi_0} H_\pi)$. Hence the assertion follows from Proposition~\ref{space_decomposition_according_to_graph}.

(ii). Let $P_{\pi_i}$ be the projection onto $H_{\pi_i}$, $i=1,2$. Projections $P_{\pi_1}$, $P_{\pi_2}$ belong to the commutant of $\psi_0(A)$ and therefore also of $\psi_0''(A'')$. Hence
$$
(\psi_1''(Q_eX)P_{\pi_1})^* (\psi_1''(Q_fX)P_{\pi_2})= \psi_0''(\langle Q_eX, Q_fX\rangle_{A''}) P_{\pi_1} P_{\pi_2}.
$$
If $[\pi_1]\neq [\pi_2]$, then $P_{\pi_1} P_{\pi_2}=0$ and in view of above $\psi_1''(Q_eX)H_{\pi_1}\,\, \bot \,\, \psi_1''(Q_fX)H_{\pi_2}$. If $[\pi_1]=[\pi_2]$ and $e\neq f$, then $\psi^{(1)}{''}(Q_e) \psi^{(1)}{''}(Q_f)\psi_1(X) H_{\pi_2}=0$ by part (i). Accordingly $\psi_1''(Q_eX)H_{\pi_1}\,\, \bot \,\, \psi_1''(Q_fX)H_{\pi_2}$.

(iii). The proof is by induction on $n$. For $n=1$ the assertion follows from (i). If the assertion holds for a certain $n \geq 1$, then it also holds for $n+1$. Indeed, it suffices to note that if $H_1$ and $H_2$ are orthogonal and $\psi_0(A)$-invariant subspaces of $H$, then so are the spaces $\psi_1(X)H_1$ and $\psi_1(X)H_2$. The latter follows because
$$
\langle \psi_1(X)H_1, \psi_1(X)H_2\rangle = \langle H_1, \psi_0(\langle X,X\rangle_A) H_2\rangle=0.
$$

(iv). Let $k=1,..., n-1$. By part (iii) we see that $\psi_k(X^{\otimes k})\psi_1''(Q_{e_n} X)...\psi_1''(Q_{e_1} X) H_\pi$ is equal to the following direct sum 
$$
\bigoplus_{(e_{n+k},...,e_1)\in s^{-(n+k)}([\pi])} \psi_1''(Q_{e_{n+k}} X)...\psi_1''(Q_{e_1} X) H_\pi \oplus 
\bigoplus_{i=0}^{k-1} \psi_i(X^{\otimes i}) K_i^{k}
$$
where $K^k_i$ is a $\psi_0(A)$-invariant subspace which has no non-zero irreducible subspaces for $\psi_0(J)$, $i=0,1,...,k-1$. Note that $K_i^{k}$ is orthogonal to $\psi_1''(Q_{e_{n-i}} X)...\psi_1''(Q_{e_1} X) H_\pi$ because the latter is irreducible for $\psi_0(A)$. This implies that $\psi_1''(Q_{e_n} X)...\psi_1''(Q_{e_1} X) H_\pi$ is orthogonal to each space $\psi_i(X^{\otimes i}) K_i^{k}$ because
$$
\langle \psi_i(X^{\otimes i}) K_i^{k}, \psi_1''(Q_{e_n} X)...\psi_1''(Q_{e_1} X) H_\pi\rangle = \langle K_i^{k}, \psi_1''(Q_{e_{n-i}} X)...\psi_1''(Q_{e_1} X) H_\pi\rangle=0.
$$
Thus it suffices to show that 
$$
\forall_{(e_{n+k},...,e_{n+1})\in s^{-k}([r(e_n)])}\qquad \psi_1''(Q_{e_{n+k}} X)...\psi_1''(Q_{e_1} X) H_\pi\,\, \bot \,\, \psi_1''(Q_{e_n} X)...\psi_1''(Q_{e_1} X) H_\pi.
$$
However, if assume that this is not true, then by part (ii) there exists $(e_{n+k},...,e_{n+1})\in s^{-k}([r(e_n)])$ such that $(e_{n+k},...,e_{k+1})=(e_n,...,e_1)$. This contradicts that $e_{k+1}\neq e_1$.
\end{proof}

\section{The uniqueness property}\label{sec:The uniqueness property}

Let $X$ be a $C^*$-correspondence over $A$, let $J$ be an ideal in $J_X$ and let $(\psi_0,\psi_1)$ be a $J$-covariant representation of $X$. It is well known, cf. for instance \cite[Corollary 11.7]{ka3}, \cite[Theorem 9.1]{kwa-doplicher} or \cite{Kak}, that if the epimorphism $\psi_0\rtimes_{J}\psi_1: \OO(J,X) \to C^* (\psi)$ is injective, then necessarily $\psi_0$ is injective and 
\begin{equation}\label{eq:strict_J__covariance0}
J=\{a\in A: \psi_0(a)\in \psi^{(1)}(\K(X))\}.
\end{equation}
In fact, $\psi_0$ is injective and satisfies \eqref{eq:strict_J__covariance0} if and only if the representation $\psi_0\rtimes_{J}\psi_1$ is faithful on the core subalgebra $\clsp\{j^{(n)}(T): T\in \K(X^{\otimes n}), n\in \N\} \subseteq \OO(J,X)$, cf. \cite[Theorem 7.3]{kwa-doplicher} or the proof in \cite{Kak}. Moreover, since we consider the case where $J\subseteq J_X$ condition \eqref{eq:strict_J__covariance0} is equivalent to 
\begin{equation}\label{eq:strict_J__covariance}
J=\{a\in J(X): \psi_0(a)=\psi^{(1)}(\phi_X(a))\},
\end{equation}
see \cite[Theorem 9.1 (i)]{kwa-doplicher}. We note that if $J=\{0\}$, then \eqref{eq:strict_J__covariance0} implies injectivity of $\psi_0$. If $J=J_X$ then \eqref{eq:strict_J__covariance} is automatically satisfied by any injective covariant representation $(\psi_0,\psi_1)$, cf. \cite[Proposition 3.3]{katsura}. We are interested in the case when these necessary algebraic conditions are also sufficient.

\begin{defn}\label{defn:uniqueness_property}
Let $X$ be a $C^*$-correspondence over $A$ and let $J$ be an ideal in $ J_X$. We say that the pair $(X,J)$ has the \emph{uniqueness property} if for any injective representation $(\psi_0,\psi_1)$ of $X$ satisfying \eqref{eq:strict_J__covariance} the map $\psi_0\rtimes_{J}\psi_1: \OO(J,X) \to C^* (\psi)$ is an isomorphism.
\end{defn}

In this section we gather certain general facts concerning the above property. We start with some useful characterizations (which are known to experts).

\begin{lem}\label{lemma on intersection property}
Let $X$ be a $C^*$-correspondence over $A$ and let $J$ be an ideal in $J_X$. The following conditions are equivalent:
\begin{itemize}
\item[(i)] The pair $(X,J)$ has the uniqueness property.
\item[(ii)] For every injective representation $\psi=(\psi_0,\psi_1)$ of $X$ satisfying \eqref{eq:strict_J__covariance} there is a circle action
$\gamma:\T \to \Aut(C^*(\psi))$ determined by $\gamma_z(\psi_0(a))=\psi_0(a)$ and $\gamma_z(\psi_1(x))=z\psi_1(x)$, for $a\in A$, $x\in X$ and $z\in \T$.
\item[(iii)] For every injective representation $\psi=(\psi_0,\psi_1)$ of $X$ satisfying \eqref{eq:strict_J__covariance} there is a linear map $E_\psi:C^*(\psi)\to \clsp\{\psi^{(n)}(T): T\in \K(X^{\otimes n}), n\in \N\}$ such that 
$$ E_\psi(\psi_n(x)\psi_{m}(y)^*)=
\begin{cases}
\psi_n(x)\psi_{n}(y)^*, & n=m,\\
0, & n\neq m.
\end{cases}
$$
(Then $E_\psi$ is necessarily a conditional expectation). 
\item[(iv)] The $C^*$-subalgebra $j_A(A)+j^{(1)}(\K(X))$ of $\OO(J,X)$ detects ideals in $\OO(J,X)$, that is, every non-zero ideal $K$ of $\OO(J,X)$ has non-zero intersection with $j_A(A)+j^{(1)}(\K(X))$.
\item[(v)] Every non-zero ideal in $\OO(J ,X)$ contains a non-zero gauge-invariant ideal.
\end{itemize}
\end{lem}

\begin{proof} 
Equivalence (i)$\Leftrightarrow$(ii) holds by \cite[Corollary 11.7]{ka3}, see also \cite[Theorem 9.1]{kwa-doplicher} or \cite{Kak}. The implication (ii)$\Rightarrow$(iii) is straightforward, as the formula $E_\psi(b)=\int_{\T}\gamma_z(b)\, dz$, $b\in C^*(\psi)$, gives the desired conditional expectation. To see that (iii) implies (i), note that for $E_\psi$ as in (iii) and $E$ the corresponding conditional expectation on $\OO(J,X)$ we have $E_\psi\circ (\psi_0\rtimes_{J}\psi_1)= (\psi_0\rtimes_{J}\psi_1)\circ E$. Thus if $a\in \ker(\psi_0\rtimes_{J}\psi_1)$, then $(\psi_0\rtimes_{J}\psi_1)(E(a^*a))=0$. But since $(\psi_0\rtimes_{J}\psi_1)$ is injective on the range of $E$, we get $E(a^*a)=0$. However, $E$ is well known to be faithful. Thus $a=0$. Therefore $\psi_0\rtimes_{J}\psi_1$ is an isomorphism. In particular $E_\psi$ is a conditional expectation. We have proved that (i)$\Leftrightarrow$(ii)$\Leftrightarrow$(iii).

If $K$ is an ideal in $\OO(J,X)$ such that $K\cap (j_A(A)+j^{(1)}(\K(X)))=\{0\}$, then composing $(j_A, j_X)$ with the quotient map $\OO(J,X)\to \OO(J,X)/K$ we get an injective representation $(\psi_0,\psi_1)$ of $X$ satisfying \eqref{eq:strict_J__covariance}. Moreover, every injective representation $(\psi_0,\psi_1)$ of $X$ satisfying \eqref{eq:strict_J__covariance} arises in this way and we have $K=\ker(\psi_0\rtimes_J\psi_1)$. This immediately gives (i)$\Leftrightarrow$(iv).

It follows from the description of gauge-invariant ideals in $\OO(J,X)$, see \cite[Proposition 11.9]{ka3}, cf. also \cite[Theorem A.4]{kwa-endo}, that gauge-invariant ideals are exactly those ideals in $\OO(J,X)$ which are generated by their intersection with $j_A(A)+j^{(1)}(\K(X))$. In particular, every ideal in $\OO(J,X)$ which intersects non-trivially $j_A(A)+j^{(1)}(\K(X))$ contains a non-zero gauge-invariant ideal. This readily gives (iii)$\Leftrightarrow$(iv).
\end{proof}

Given two Morita equivalent $C^*$-correspondences $X$ and $Y$ with the equivalence given by an $A$-$B$ bimodule $M$, it is proved in \cite[Proposition 4.2]{ekk} that the Rieffel correspondence $A \triangleright J\longmapsto M(J):=\langle M, JM \rangle_B\triangleleft B$ restricts to an order bijection between the ideals of $J_X$ and $J_Y$. Moreover, the corresponding relative Cuntz--Pimsner algebras $\OO(J,X)$ and $\OO(M(J),Y)$ are Morita equivalent, see \cite[Theorem 4.4]{ekk}. An inspection of the proof shows that this Morita equivalence sends gauge invariant ideals of $\OO(J,X)$ to gauge invariant ideals of $\OO(M(J),Y)$. Thus modulo Lemma~\ref{lemma on intersection property} we get that uniqueness property is preserved under Morita equivalence. Namely, we have:

\begin{prop}\label{me_gauge}
Let $X$ and $Y$ be two $C^*$-correspondences over $A$ and $B$ respectively, and let $J$ be an ideal of $J_X$. Suppose that $X$ and $Y$ are Morita equivalent via an equivalence $A$-$B$-bimodule $M$. Then the pair $(X,J)$ has the uniqueness property if and only if $(Y,M(J))$ has the uniqueness property.
\end{prop}

Recall that given a positively $X$-invariant ideal $I$ in $A$ the space $IX$ is naturally a $C^*$-correspondence over $I$. The uniqueness property is preserved under restrictions to positively invariant ideals:

\begin{prop}\label{prop:inducing non-detections}
Let $X$ be a $C^*$-correspondence over a $C^*$-algebra $A$, and let $J$ be an ideal of $J_X$. If $(J,X)$ has the uniqueness property, then $(I\cap J,IX)$ has the uniqueness property for every positively $X$-invariant ideal $I$ in $A$.
\end{prop}

\begin{proof} Suppose that there is a positively $X$-invariant ideal $I$ such that $(J\cap I,IX)$ does not have the uniqueness property. Equivalently (see Lemma~\ref{lemma on intersection property}), there is a non-zero ideal $L$ in $\OO(J\cap I ,IX)$ which does not contain any non-zero gauge-invariant ideal. By the proof of \cite[Theorem A.4]{kwa-endo}, see also the proof of \cite[Proposition 9.3]{ka3}, we may identify $\OO(J\cap I ,IX)$ with the $C^*$-subalgebra of $\OO(J ,X)$ generated by $j_A(I)$ and $j_X(IX)$, and then $\OO(J\cap I ,IX)=j_A(I)\OO(J ,X)j_A(I)$ is the hereditary subalgebra of $\OO(J ,X)$ generated by $j_A(I)$. Thus $\OO(J\cap I ,IX)$ and the ideal $\OO(J ,X) j_A(I) \OO(J ,X)$, generated by $j_A(I)$, are Morita equivalent. Moreover, the equivalence bimodule $j_A(I)\OO(J ,X)$ respects the gauge actions in the $C^*$-algebras $\OO(J\cap I ,IX)=j_A(I)\OO(J ,X)j_A(I)$ and $\OO(J ,X) j_A(I) \OO(J ,X)$ (they are gauge-invariant subalgebras of $\OO(J,X)$). Thus we have a lattice isomorphism $K\mapsto K\cap \OO(J ,X) j_A(I) \OO(J ,X)$ from the set of ideals in $\OO(J\cap I ,IX)$ to that of $\OO(J ,X) j_A(I) \OO(J ,X)$, which restricts to a lattice isomorphism between the gauge-invariant ideals. Hence, if $K$ is an ideal in $\OO(J ,X) j_A(I) \OO(J ,X)$ corresponding to the ideal $L$ in $\OO(J\cap I ,IX)$, then $K$ is a non-zero ideal in $\OO(J ,X)$ and does not contain a non-zero gauge-invariant ideal. Accordingly, $(J,X)$ does not have the uniqueness property by Lemma~\ref{lemma on intersection property}.
\end{proof}

\section{The uniqueness theorem}\label{sec:The uniqueness theorem}

The aim of this section is to give sufficient conditions for a uniqueness property of the relations defining the relative Cuntz--Pimsner algebra $\OO(J,X)$. In view of Lemma~\ref{lem:stupid_lemma}, to define (strong) topological freeness of a graph $E$ on a given set $U$, it suffices to consider the restriction of $E$ to the set $U\cup E^{-1}(U)$. This motivates the following definition.

\begin{defn}\label{defn:topological_free_correspondence}
Let $X$ be a $C^*$-correspondence over a $C^*$-algebra $A$, and let $J$ be an ideal of $J_X$. Suppose that $X(J)$ is of Type I, so that the ideal $K:=J+X(J)$ is of Type I, by Lemma~\ref{lem:stanislaw}, and hence the dual graph $E_{KXK}$ to the $C^*$-correspondence $KXK$ over $K$ is well defined. We say that $X$ is \emph{topologically free on $J$} if $E_{KXK}$ is topologically free on $\widehat{J}\subseteq \widehat{K}$. Similarly, we say that $X$ is \emph{strongly topologically free on $J$} if $E_{KXK}$ is strongly topologically free on $\widehat{J}$.
\end{defn}

\begin{prop}\label{prop:strong_vs_normal}
If $X(J)$ is liminal, then $X$ is strongly topologically free on $J$ if and only if $X$ is topologically free on $J$.
\end{prop}

\begin{proof}
Combine Lemma~\ref{lem:topological_freeness} and Corollary~\ref{continuity of dual graphs}.
\end{proof}

\begin{thm}[Uniqueness theorem]\label{uniqueness theorem}
Let $X$ be a $C^*$-correspondence over a $C^*$-algebra $A$ and let $J$ be an ideal in $J_X$. Suppose that either
\begin{itemize}
\item[A1)] the dual multivalued map $\widehat{X}$ is weakly topologically aperiodic on $\widehat{J}$, or 
\item[A2)] $X(J)$ is of Type I and $X$ is strongly topologically free on $J$. 
\end{itemize} 
Then the pair $(X,J)$ has the uniqueness property.
\end{thm}

The proof of Theorem 7.3 is given at the end of this section. For the proof, we need a couple of lemmas and a proposition.

We fix a $C^*$-correspondence $X$ over a $C^*$-algebra $A$, an ideal $J$ of $J_X$ and an injective representation $\psi=(\psi_0,\psi_1)$ of $X$ satisfying \eqref{eq:strict_J__covariance}. We start with the analysis of representations of certain subalgebras of $C^*(\psi)$. For each $n\in \N$ we define the following $C^*$-subalgebras of $C^*(\psi)$: 
$$ 
\K_n:= \psi^{(n)}(\K(X^{\otimes n})),\quad \K^J_n:= \psi^{(n)}(\K(X^{\otimes n}J)),\quad \F_n:=\sum_{i=0}^n \K_i, \quad \F_\infty:=\sum_{i=0}^\infty \K_i. 
$$
We call $\F_\infty$ the \emph{core subalgebra of} $C^*(\psi)$. The following lemma is well-known to experts, cf. \cite[Section 5]{katsura}.

\begin{lem}\label{katsura's lemma}
Let $X$ be a $C^*$-correspondence over a $C^*$-algebra $A$ and let $J$ be an ideal in $J_X$. We have the following properties:
\begin{itemize}
\item[(i)] If $a\in \K(X^{\otimes n}J)$, then $a\otimes 1_X\in \K(X^{\otimes n+1})$ and $\psi^{(n)}(a)=\psi^{(n+1)}(a\otimes 1_X)$.
\item[(ii)] $\K_n\cap \K_{n+1}=\K^J_n$ for every $n\in \N$;
\item[(iii)] the sum $\sum_{k=i}^n \K_k$ is an ideal in $\F_n$ for every $0\leq i \leq n$.
\end{itemize}
\end{lem}

\begin{proof}
(i). It follows, for instance, from \cite[Proposition 3.22]{kwa-doplicher}.

(ii). This is \cite[Proposition 5.9]{katsura}.

(iii). It follows from \cite[Lemma 5.4]{katsura}.
\end{proof}

For every $n\geq 0$, the $*$-homomorphism $\psi^{(n)}:\K(X^{\otimes n})\rightarrow C^*(\psi)$ is injective. Therefore, $\psi^{(n)}$ induces $*$-isomorphisms $\K(X^{\otimes n})\cong\K_n$ and $\K(X^{\otimes n}J)\cong \K_n^J$, and homeomorphisms $\widehat{\K(X^{\otimes n})}\cong\widehat{\K_n}$ and $\widehat{\K(X^{\otimes n}J)}\cong \widehat{\K_n^J}$. We will identify these topological spaces via these homeomorphisms.

\begin{lem}\label{embedding of spectra2} 
Given $\pi \in \Irr(\F_n)$ with $\pi(\K_n^J)\neq 0$ and any extension $\tilde{\pi}\in \Irr(\F_{n+1})$ of $\pi$, we have $\tilde{\pi}(\K_{n+1})\neq 0$. Moreover, the following diagram of multivalued maps commutes
$$
\xymatrix{
\widehat{\K_n^J} \ar[d]_{\widehat{\psi^{(n)}}}^{\cong} & & \ar[ll]_{\widehat{\iota} }\widehat{\K_{n+1}} 
\ar[d]^{\widehat{\psi^{(n+1)}}}_{\cong}
\\
\widehat{\K(X^{\otimes n}J)} &\qquad & \widehat{ \K(X^{\otimes n+1})}, \ar[ll]_{\widehat{\otimes 1_X}\,\,\,\,\,}
}
$$
where $\iota:\K_n^J \to \K_{n+1}$ is the inclusion, $\otimes 1_X:\K(X^{\otimes n}J) \stackrel{}{\longrightarrow}\K(X^{\otimes n+1})$ is an embedding by Lemma~\ref{lemma on tensoring regular correspondences}, and $\widehat{\K_n^J} \subseteq \widehat{\F_{n}}	 $ and $ \widehat{\K_{n+1}}	\subseteq \widehat{\F_{n+1}}$ are open subsets.
\end{lem}

\begin{proof} By Lemma~\ref{katsura's lemma} (i), the following diagram commutes
$$
\xymatrix{ 
\K_n^J \ar[rr]^{\iota } & & \K_{n+1} 
\\
\K(X^{\otimes n}J) \ar[u]^{\psi^{(n)}}_{\cong} \ar[rr]^{\otimes 1_X\,\,\,\,\,} &\qquad & \K(X^{\otimes n+1}). \ar[u]_{\psi^{(n+1)}}^{\cong} \\
}
$$
By Lemma~\ref{katsura's lemma} (iii), $\K_n$ is an ideal in $\F_n$. Since $\K_n^J$ is an ideal in $\K_n$, also $\K_n^J$ is an ideal in $\F_n$. In particular, $\widehat{\K_n^J} \subseteq \widehat{\F_{n}}	 $ and $ \widehat{\K_{n+1}}	\subseteq \widehat{\F_{n+1}}$ are open subsets. This readily implies the assertion. 
\end{proof}

\begin{lem}\label{embedding of spectra} 
Given $\pi \in \Irr(\F_n)$ such that $\pi(\K_{n}^J)=0$, there exists $\tilde{\pi}\in \Irr(\F_{n+1})$ that extends $\pi$ such that $\tilde{\pi}(\K_{n+1})=0$. 
\end{lem}

\begin{proof}
Since $\pi(\K_{n}^J)=0$, there is an irreducible representation $\overline{\pi}:\F_n/\K_{n}^J\to B(H_\pi)$ such that $\overline{\pi}(x+\F_n)=\pi(x)$ for $x\in\F_n$. Let $q:\F_{n+1}\to\F_{n+1}/\K_{n+1}$ be the quotient map, and let $\Phi:\F_{n+1}/\K_{n+1}\to\F_n/(\F_n\cap\K_{n+1})$ be the canonical isomorphism. By \cite[Proposition 5.12]{katsura}, $\F_n\cap\K_{n+1}=\K_{n}^J$, so $\tilde{\pi}:=\overline{\pi}\circ\Phi\circ q$ is an irreducible representation of $\F_{n+1}$ that extends $\pi$. Moreover, $\tilde{\pi}(\K_{n+1})=0$.
\end{proof}

The following proposition is the main ingredient in the proof of Theorem~\ref{uniqueness theorem}.

\begin{prop}\label{2.6} 
Retain the assumptions of Theorem~\ref{uniqueness theorem}. Let $n>0$ and $m>0$. For each $a \in \F_n$ and each $\varepsilon >0$ there exist a representation $\nu: C^* (\psi)\to \B(H_\nu)$ and contractions $Q_1,Q_2 \in \B(H_\nu)$, such that 
\begin{itemize}
\item[1)] $ \Vert Q_1\nu (a) Q_2\Vert \ge \Vert a \Vert - \varepsilon$,
\item[2)] $Q_1\nu(\psi_{k}(X^{\otimes k}) \F_n)\, Q_{2 } = 0$ and $Q_1\nu(\F_n \psi_{k}(X^{\otimes k})^*)\, Q_{2 } = 0$ for $k=1,2,\dots,m$.
\end{itemize}
\end{prop}

\begin{proof}
Let $\varepsilon>0$ and $a\in \F_n$, then the functional $\widehat{\F_n}\rightarrow \C$ defined by $[\pi]\to \|\pi(a)\|$ is lower semicontinuous, and attains its upper bound equal to $\|a\|$ (see for instance \cite[App. A]{morita}). Accordingly, there exists a non-empty open set $U\subseteq \widehat{\F_n}$ such that 
$$
\|\pi(a)\|>\|a\|-\varepsilon\qquad \text{for every }[\pi]\in U.
$$
Let us consider the following two cases.

\textbf{Case 1).} Assume that there exists an irreducible representation $\pi:\F_n\rightarrow\mathcal{B}(H_\pi)$ with $[\pi]\in U$, for which there exists an extension $\pi_N\in \Irr(\F_{n+N})$, for some $N\geq 0$, such that $[\pi_N]$ belongs to $\widehat{\F_{n+N}}\setminus \widehat{\K_{n+N}}$ (i.e. $\pi_N(\K_{n+N})=0$). Let $i$ be the maximal number with $0\leq i\leq n+N$ such that $\pi_N(\K_i)\neq 0$. Thus, $\pi_N(\sum_{k=i+1}^{n+N}\K_k)=0$, and hence $[\pi_N]\in \widehat{\K_i}\setminus \widehat{\sum_{k=i+1}^{n+N}\K_k}$. In particular, $\pi_N:\F_{n+N}\rightarrow\mathcal{B}(H_{\pi_N})$ restricts to an irreducible representation of $\K_i$, and hence $\overline{\pi_N(\K_i)H_{\pi_N}}=H_{\pi_N}$.

Now given any extension $\nu:C^*(\psi)\rightarrow \mathcal{B}(H_\nu)$ of $\pi_N$, $\bar{\psi}:=(\nu\circ \psi_0,\nu\circ\psi_1)$ is a representation of $X$ on $H_\nu$.

Let $Q_{\pi_N}:H_\nu\rightarrow H_{\pi_N}$ be the orthogonal projection onto $H_{\pi_N}$, and for each $j>0$, let $P_j:H_\nu\rightarrow\overline{\nu(\K_j)H_\nu}$ be the orthogonal projection on $\overline{\nu(\K_j)H_\nu}$. Observe that $\overline{\nu(\K_j)H_\nu}=\overline{\nu(\psi_j(X^{\otimes j}))H_\nu}$ (cf. Lemma~\ref{Fowler-Reaburn lemma}). Because $\nu$ is an extension of $\pi_N$ we have that $H_{\pi_N}=\overline{\pi_N(\K_i)H_{\pi_N}}\subseteq \overline{\nu(\K_i)H_{\nu}}$, so $Q_{\pi_N}\leq P_i$. Moreover, since $\pi_N$ vanishes on the ideal $\sum_{k=i+1}^{n+N}\K_k$, it follows that $Q_{\pi_N}P_{i+1}=P_{i+1}Q_{\pi_N}=0$. Hence 
\begin{equation}\label{subprojection of the difference}
Q_{\pi_N}\leq (P_{i}-P_{i+1}).
\end{equation}
It follows from Equation \eqref{essential projections versus fibers} in Lemma~\ref{Fowler-Reaburn lemma} that $\nu(\psi_k(X^{\otimes k}))P_j=P_{k+j}\nu(\psi_k(X^{\otimes k}))P_j$ for every $j,k\geq 0$. Therefore, for $k>0$ we get 
$$
(P_i-P_{i+1})\nu(\psi_k(X^{\otimes k}))(P_i-P_{i+1})=0,
$$ 
and hence by \eqref{subprojection of the difference} it follows that $Q_{\pi_N}\nu(\psi_k(X^{\otimes k}))Q_{\pi_N}=0$. Let $Q_\pi\in \mathcal{B}(H_\nu)$ be the projection onto $H_\pi$. Since $Q_\pi\leq Q_{\pi_N}$ we get $Q_{\pi}\nu(\psi_k(X^{\otimes k}))Q_{\pi}=0$ for every $k>0$. In particular, putting $Q_1=Q_2=Q_\pi\in \nu(F_n)'$, we get $ \Vert Q_1\nu (a) Q_2\Vert =\|\pi(a)\|\ge \Vert a \Vert - \varepsilon$, and 
$$
Q_1\nu(\psi_{k}(X^{\otimes k}) \F_n)\, Q_{2 } =Q_\pi\nu(\psi_{k}(X^{\otimes k})) Q_2 \nu(\F_n)=0,
$$
and similarly $Q_1\nu(\F_n \psi_{k}(X^{\otimes k})^*)\, Q_{2 } = 0$, for every $k>0$, as desired.

\textbf{Case 2).} Assume that for every $\pi\in \Irr(\F_n)$ with $[\pi]\in U$, every $k\geq 0$ and every irreducible extension $\pi_k:\F_{n+k}\rightarrow \mathcal{B}(H_k)$ of $\pi$ we have $\pi_k(\K_{n+k})\neq 0$. It follows from Lemma~\ref{embedding of spectra} that given any $\pi\in \Irr(\F_n)$ with $[\pi]\in U$ and every extension $\pi_k\in \Irr(\F_{n+k})$, $k\geq 0$, of $\pi$ we have $\pi_k (\K_{n+k}^J)\neq 0$ that is $[\pi_k]\in \widehat{\K_{n+k}^J}$. In particular, $U\subseteq \widehat{\K^J_n}$.

\textbf{Claim 1.} There are $\xi$, $\eta \in X^{\otimes n}$ with $\|\xi\|=\|\eta\|=1$ and $\|\psi_n(\xi)^*a \psi_n(\eta)\|> \|a\|-\varepsilon/2$.

\medskip

\begin{Proof of}{Claim 1} Since $\K_n$ is an ideal in $\F_n$, $a$ acts on $\K_n$ as a multiplier, say $m(a)$. The inclusion $U\subseteq \widehat{\K_n}$ implies that $\|m(a)\|=\|a\|$. Indeed, we may identify multipliers of $\K_n$ with operators acting on the space of atomic representation of $\K_n$, and then there is $[\pi]\in U$ such that $\|a\|=\|\pi(a)\|\leq \|m(a)\|$. Moreover, the algebra of multipliers of $\K_n$ is isomorphic to $\LL(X^{\otimes n})$, with the isomorphism $\Psi:=\overline{\psi^{(n)}}^{-1}$ where $\overline{\psi^{(n)}}$ is the unique extension of the isomorphism $\psi^{(n)}:\K(X^{\otimes n})\to \K_n$. This implies that 
\begin{align*}\|a\|&=\|\Psi(m(a))\|=\sup\{ \|\langle \xi, \Psi(m(a)) \eta\rangle_A\|:\xi, \eta \in X^{\otimes n}, \|\xi\|=\|\eta\|=1\} \\
&=\sup\{ \|\psi_n(\xi)^* \psi_n(\Psi(m(a))\eta)\|:\xi, \eta \in X^{\otimes n}, \|\xi\|=\|\eta\|=1\} \\
&=\sup\{ \|\psi_n(\xi)^*a \psi_n(\eta)\|:\xi, \eta \in X^{\otimes n}, \|\xi\|=\|\eta\|=1\}.
\end{align*}
\end{Proof of}

Let $\xi$, $\eta \in X^{\otimes n}$ be as in Claim 1 and put $g:=\psi_n(\xi)^*a \psi_n(\eta)$. It follows that to prove the assertion it suffices to show the corresponding assertion for $g\in \psi_0(A)=\F_0$ instead of $a\in F_n$. In fact we will find $\nu: C^* (\psi)\to \B(H_\nu)$ and a projection $Q\in \B(H_\nu)$, such that 
$$
\Vert Q\nu (g) Q\Vert \ge \Vert g \Vert - \varepsilon/2,\qquad Q\nu(\psi_{k}(X^{\otimes k}))\, Q = 0 \text{ for $k=1,2,\dots,m$}.
$$ 
Then $Q_1:=Q\nu(\psi_n(\xi))^*$ and $Q_2:= \nu(\psi_n(\eta))Q$ will fulfill the desired conditions for $a\in \F_n$.

Accordingly, we let $V:=\{[\pi] \in \SA:\|\pi(\psi_0^{-1}(g))\|>\|g\|-\varepsilon/2\}$. This is an open nonempty subset of $\SA$. Case 1) for $g$ (instead of $a$) gives the assertion. Thus we may assume that we are in Case 2. Namely, we assume that for every $\pi\in \Irr(\psi_0(A))$ with $[\pi\circ \psi_0]\in V$ and every irreducible extension $\pi_k:\F_{k}\rightarrow \mathcal{B}(H_k)$, $k\geq 0$, of $\pi$ we have $[\pi_k]\in \widehat{\K_{k}^J}$. In particular, $V\subseteq \widehat{J}$.

In the following claim and below we treat the multivalued map $\SA$ and its restriction to $\widehat{JXJ}$ as directed graphs (with no multiple edges).

\textbf{Claim 2.} For every sequence $\pi_0\leq \pi_1 \leq \ldots \leq \pi_l$ where $\pi_k\in \Irr(\F_{k})$, $l>0$, and $[\pi_0\circ \psi_0]\in V$ there is a path $([v_0,v_1]$,$\ldots$,$[v_{l-1},v_l])$ in $\widehat{X}$ (treated as a graph) such that
\begin{equation}\label{relation between representations}
[\pi_k\circ\psi^{(k)} ]=[X^{\otimes k}\dashind]([v_k]),\quad \text{ for all } k =0,...,l.
\end{equation}
Moreover, for every path $([v_0,v_1]$,$\ldots$,$[v_{l-1},v_l])$ in $\widehat{X}$ with $v_0\in V$ there is a sequence $\pi_0:=v_0\circ \psi_0^{-1}\leq \pi_1 \leq \ldots \leq \pi_l$ where $\pi_k\in \Irr(\F_{k})$ and \eqref{relation between representations} holds. 
In particular, $([v_0,v_1]$,$\ldots$,$[v_{l-1},v_l])$, is necessarily a path in $\widehat{JXJ}$, i.e. $[v_k]\in \widehat{J}$ for every $k =0,...,l$.

\medskip

\begin{Proof of}{Claim 2} Let $\pi_0\leq \pi_1 \leq \ldots \leq \pi_l$ where $\pi_k\in \Irr(\F_{k})$, $l>0$, and $[\pi_0\circ \psi_0]\in V$. By our assumption we have $[\pi_k]\in \widehat{\K_{k}^J}$, for $k=0,...,l$. By Lemma~\ref{embedding of spectra2}, putting $\sigma_k:=\pi_k\circ \psi^{(k)}$ we get $[\sigma_k]\in \widehat{\K(X^{\otimes k} J)}$, for $k=0,...,l$, and $[\sigma_{k}]\in \widehat{\otimes 1_X}[\sigma_{k+1}]$ for $k=0,...,l-1$. Hence by Proposition~\ref{tensoring vs duals} there is a path $([v_0,v_1]$,$\ldots$,$[v_{l-1},v_l])$ in $\widehat{X}$ such that $[\sigma_k]=[X^{\otimes k}\dashind]([v_k])$, for $k=0,...,l$, so \eqref{relation between representations} is satisfied. In particular, cf. Proposition~\ref{tensoring vs duals}, we have $[v_k]\in \widehat{\langle X^{\otimes k}J, X^{\otimes k }J\rangle}_A\subseteq \widehat{J}$ for every $k =0,...,l$.

Conversely, let $([v_0,v_1]$,$\ldots$, $[v_{l-1},v_l])$ be a path in $\widehat{X}$ with $v_0\in V$. We put $\sigma_0=v_0$. Since $v_0\in \widehat{J}$, Proposition~\ref{tensoring vs duals} implies that there is $\sigma_1\in \Irr\LL(X)$ such that $[\sigma_1]\in \widehat{\K(X)}$ and $\sigma_0$ is equivalent to a subrepresentation of $\sigma_1\circ \phi_X$. Hence by Lemma~\ref{embedding of spectra2}, there is an extension of $\pi_0:=v_0\circ\psi_0^{-1}\in \Irr(\F_0)=\Irr(\psi_0(A))$ to $\pi_1\in \Irr(\F_1)$. Since $v_0\in V$ we get $[\pi_1]\in \widehat{\K_{1}^J}$ by our standing assumption. In particular, $[\pi_1\circ\psi^{(1)} ]=[X\dashind]([v_1])$ and hence $v_1\in \widehat{J}$. Thus we may apply the same argument to $v_1$ instead of $v_0$. Repeating this argument inductively, we get the desired sequence $\pi_0\leq \pi_1 \leq \ldots \leq \pi_l$. 
\end{Proof of}

\textbf{Subcase 2a).} Suppose that A1) in Theorem~\ref{uniqueness theorem} holds. By Claim 2 every $v\in V$ is an end of an infinite path in $\widehat{JXJ}$. Since $\widehat{X}$ is weakly topologically aperiodic on $\widehat{J}$, there is a path $([v_0,v_1],\ldots [v_{m-1},v_m])$ such that $[v_0]\in V$ and $[v_0]\neq [v_k]$ for $k=1,\ldots,m$. Let $\pi_0\leq \pi_1 \leq \ldots \leq \pi_m$, $\pi_k\in \Irr(\F_{k})$, be the corresponding sequence satisfying \eqref{relation between representations} as described in Claim 2.			Let $\nu:C^*(\psi)\rightarrow \mathcal{B}(H_\nu)$ be any extension of $\pi_m$. Let $k=1,\ldots,m$. We denote by $P_k:H_\nu\rightarrow H_k$ the orthogonal projection onto the space
$$
H_k:=\nu\big(\psi_{k}(X^{\otimes k})^*\big)H_{\pi_k}.
$$
By Lemma~\ref{representation structure lemma2} the formula $\nu_k(b)=\nu(\psi_0(b))|_{H_k}$, $b\in A$, defines a representation $\nu_k:A \rightarrow \mathcal{B}(H_k)$ and either $H_k=\{0\}$ or $\nu_k\in \Irr(A)$ and then
$$
[\nu_{k}] = [X^{\otimes k}\dashind]^{-1} ([\pi_k\circ\psi^{(k)} ])=[X^{\otimes k}\dashind]^{-1} [X^{\otimes k}\dashind]\dashind]([v_k])=[v_k].
$$ In particular, for every $k=1,...,m$, $P_k\in \nu(A)'$ and either $\nu_k=0$ or $[\nu_k]=[v_k]\neq [v_0]=[\nu_0]$. Since two inequivalent irreducible subrepresentations are disjoint, they act on orthogonal subspaces. Hence, we get 
\begin{equation}\label{relation between bla bla}
P_{0}P_k=P_kP_{0}=0, \qquad k=1,...,m. 
\end{equation}
Moreover, since $H_{\pi_k}=\psi^{(k)}(\K(X^{\otimes k}))H_{\pi_0}$ and $\psi_{k}(X^{\otimes k})^* \psi^{(k)}(\K(X^{\otimes k}))=\psi_{k}(X^{\otimes k})^* $, for each $k=1,...,m$ we have
$$
H_k= \nu(\psi_{k}(X^{\otimes k})^*)H_{\pi_0}.
$$
Combining this with \eqref{relation between bla bla}, we get $P_0 \nu(\psi_{k}(X^{\otimes k})^*)P_0=0$. Moreover, we have $ \Vert P_0\nu (g) P_0\Vert =\| \nu_0(g)\|=\| v_0(\psi_0^{-1}(g))\|\ge \Vert g \Vert - \varepsilon/2$. Thus putting $Q:=P_0$ we get the assertion.

\medskip

\textbf{Subcase 2b).} Suppose that A2) of Theorem~\ref{uniqueness theorem} holds. By Claim 2, every $v\in V$ is an end of an infinite path in the graph $E_{Y}$ dual to the $C^*$-correspondence $Y=JXJ$ over $J$. By strong topological freeness there is a path $(e_{l},...,e_1)\in r^{-l}(V)$ in $E_{Y}$ with $l > m$ such that $e_k\neq e_1$ for every $k=2,...,l$. In particular, $([r(e_{l}),s(e_{l})]$,$\ldots$,$[r(e_1),s(e_1)])$ is a path in $\widehat{Y}$. Let $\pi:=s(e_1)\circ \psi_0^{-1}$ be the irreducible representation of $\F_0=\psi_0(A)$ and extend it (in an arbitrary way) to a representation $\nu: C^* (\psi)\to \B(H_\nu)$. Let $P_\pi\in \nu(\F_0)$ be the projection onto $H_{\pi}\subseteq H_\nu$. Let $Q: E_{Y}^1\to \K(X)''$ be the mapping as in Proposition~\ref{space_decomposition_according_to_graph}. Let	$\nu'': C^* (\psi)''\to \B(H_\nu)$ be the weakly continuous extension of $\nu$. Using inductively Lemma~\ref{lem:proposition_in_terms_of_representations}(i) we get that the space
$$
H_Q:=	\nu''(\psi_1''(Q_{e_{l}}X) ... \psi_1''(Q_{e_2}X) \psi_1''(Q_{e_1}X))H_{\pi}
$$		is irreducible for $\nu(\F_0)$ and the equivalence class of the corresponding representation of $\F_0$ is $r(e_l)\in V$. Accordingly, if we let $Q$ to be the projection onto $H_Q$, then $\|\nu(g)Q\|\geq \|g\|- \varepsilon/2$. By Lemma~\ref{lem:proposition_in_terms_of_representations}(iv) we have
$$
\nu(\psi_k(X^{\otimes k}))H_Q \bot H_Q \qquad \text{ for }k=1,...,l-1.
$$
Hence $Q\nu(\psi_k(X^{\otimes k}))Q=0$ for $k=1,...,m\leq l-1$. 
\end{proof}

\begin{proof}[Proof of Theorem~\ref{uniqueness theorem}] 
By Lemma~\ref{lemma on intersection property} it suffices to prove that there is a conditional expectation $E$ from $C^*(\psi)$ onto $\F_\infty$ which sends elements from $\psi_m(X^{\otimes m})\F_n$, $n\ge 0$, $m\ge 1$, to zero. To this end, it suffices to show that for every element of the form
$$
b= \sum_{k=1}^{m} a_{-k}^* + a_0 + \sum_{k=1}^{m} a_{k}
$$
where $a_{\pm k}\in \psi_k(X^{\otimes k})\F_n$, $k=1,...,m$, the following inequality holds $ \|a_0\|\leq \|b\|. $ This follows immediately from Proposition~\ref{2.6}.
\end{proof}

\section{Conditions necessary for uniqueness property}\label{sec:Conditions necessary for uniqueness property}

Let $X$ be a $C^*$-correspondence over $A$ and let $J$ be an ideal of $J_X$. In this section we prove the converse to Theorem~\ref{uniqueness theorem} in case A2). Moreover, we extend it using a fairly algebraic property called \emph{acyclicity} whose definition is inspired by the condition introduced in \cite[Definition 4.1]{Carlsen_Ortega_Pardo}. This property could also be considered a version of pure outerness for $C^*$-correspondences, cf. \cite[Definition 4.3]{KM} and \cite[Definition 3.7]{Schweizer}.

\begin{defn}\label{def:cycling}
Let $X$ be a $C^*$-correspondence over $A$ and $I$ a positively $X$-invariant ideal of $A$. We say $X$ is \emph{cyclic with respect to $I$} if there exists $n\in\mathbb{N}$ such that $(IX)^{\otimes n}$ is isomorphic to the trivial $C^*$-correspondence $I$, i.e. there exists a bijective map $\Psi:(IX)^{\otimes n}\rightarrow I$ such that $a\Psi(x)b=\Psi(\phi_X(a)xb)$ and $\Psi(\langle x,y\rangle_{I})=\Psi(x)^*\Psi(y)$ for every $x,y\in (IX)^{\otimes n}$ and $a,b\in I$ (then $\Psi$ is necessarily isometric). In this case we say that $IX$ has \emph{period $n$}. Given an ideal $J$ of $J_X$ we say that $X$ is \emph{$J$-acyclic} if there are no non-zero ideals $I$ in $J$ such that $X$ is cyclic with respect to $I$.
\end{defn}

\begin{rem}\label{rem:cyclicity}
It follows readily from the definition that if $X$ is a cyclic with respect to $I$, then $I\subseteq (\ker \phi_X)^\bot$, $IX=IXI$ and $\langle IX,IX\rangle=I$. 
\end{rem}

The aim of the present section is to prove the following theorem.

\begin{thm}\label{thm:acyclicity}
Let $X$ be a $C^*$-correspondence over a $C^*$-algebra $A$, and let $J$ be an ideal of $J_X$. Consider the following conditions:
\begin{enumerate}
\item \label{en:uniqueness} the pair $(X,J)$ has the uniqueness property;
\item \label{en:acyclicity} $X$ is $J$-acyclic.
\end{enumerate}
Then \eqref{en:uniqueness}$\Rightarrow$\eqref{en:acyclicity}. If in addition $X(J)$ is liminal, then \eqref{en:uniqueness}$\Leftrightarrow$\eqref{en:acyclicity}$\Leftrightarrow$\textup{(3)} where 
\begin{enumerate}
\item[(3)] $X$ is topologically free on $J$.
\end{enumerate}
\end{thm}

We start with facts leading to the proof of the implication \eqref{en:uniqueness}$\Rightarrow$\eqref{en:acyclicity}.

\begin{lem}\label{lem:cyclicy_implies_equivalence}
Suppose that $X$ is cyclic with respect to a positively invariant ideal $I$ of $J_X$. Then the $C^*$-correspondence $IX$ is an equivalence Hilbert bimodule.
\end{lem} 

\begin{proof}
Suppose that $IX$ has period $n$. By Remark~\ref{rem:cyclicity}, $\langle IX,IX\rangle=I$, i.e. $X$ is full on the right, and also we have $I\subseteq (\ker \phi_X)^\bot$. The latter implies that the homomorphism $\phi_{IX}=\phi_{X}|_{IX}:I\to \K(IX)=\phi_X(I)\K(X)\phi_X(I)$ is injective. Thus it suffices to show that $\phi_{IX}:I\to \K(IX)$ is surjective, cf. \cite[3.3]{ka2} or \cite[Proposition 1.11]{kwa-doplicher}. To this end, let us consider a universal representation $(j_0,j_{1})$ of $(I,IX)$ into $\OO(I,IX)$. Since $I\subseteq J_{IX}$, this representation is injective. Note that we have
$$
I=\{a\in I: j_0(a)=j^{(1)}(\phi_{IX}(a))\}=\{a\in I: j_0(a)\in j^{(1)}(\K(IX))\},
$$
cf. \cite[Corollary 11.7]{ka3} or \cite[Theorem 9.1]{kwa-doplicher}. Moreover, for the associated representation $j^{(n)}:\K((IX)^{\otimes n})\to \OO(I,IX)$ we also have
$$
I=\{a\in I: j_0(a)=j^{(n)}(\phi_{(IX)^{\otimes n}}(a))\}=\{a\in A: j_0(a)\in j^{(n)}(\K((IX)^{\otimes n}))\},
$$
see \cite[Lemma 3.6.]{kwa-szym}. Since $(IX)^{\otimes n}\cong I$ is an equivalence bimodule, $\phi_{(IX)^{\otimes n}}:I\to \K((IX)^{\otimes n})$ is bijective and we conclude that 
$$
j_0(I)=j^{(n)}(\K((IX)^{\otimes n})).
$$
Multiplying this equality by $j_n((IX)^{\otimes n})^*$ from left and by $j_n((IX)^{\otimes n})$ from right, and using positive invariance of $I$ we get
$$
j_0(I)=j^{(1)}(\K(IX)).
$$
Hence, for every $T\in \K(IX)$ there is $a\in I$ such that $j^{(1)}(\phi_{IX}(a))=j^{(1)}(T)$. Since $j^{(1)}$ is injective, this implies that $\phi_{IX}(a))=T$. Thus, $\phi_{IX}$ is surjective.
\end{proof}

\begin{lem}\label{lem:periodicity_implies_non-uniqueness}
Suppose $X$ is a periodic equivalence bimodule over $A$, i.e. there is $n>0$ such that $X^{\otimes n}\cong A$. Then the pair $(A,X)$ does not have the uniqueness property.
\end{lem}

\begin{proof}
We may identify the algebra $\OO(A,X)$ with the crossed product $A\rtimes_X \Z$ introduced in \cite{AEE}, cf. \cite[Proposition 3.7]{ka2}. We need to show that $A$ does not detect ideals in $A\rtimes_X \Z$. Using Takai duality we may reduce our considerations to a classical crossed product by $\Z$. Indeed, there is an automorphism $\alpha:B\to B$ and an equivalence $A$-$B$ bimodule $M$ such that $X\cong M\otimes B_\alpha \otimes M^*$ where $B_\alpha$ is the Hilbert bimodule associated to $\alpha$, and $A$ does not detect ideals in $A\rtimes_X \Z$ if $B$ does not detect ideals in $B\rtimes_\alpha \Z$. However, we have natural isomorphisms 
$$B_{\alpha^n}\cong B_{\alpha}^{\otimes n}\cong (M^*\otimes X \otimes M)^{\otimes n}\cong (M^*\otimes X^{\otimes n} \otimes M)\cong M^*\otimes M\cong B.
$$
This implies that $\alpha^n=id$. Using this one readily constructs a covariant representation $(\pi,U)$ of $(B,\alpha)$ with $\pi$ injective $U$ such that $U^n=1$. Then the kernel of $\pi\rtimes_\alpha U$ does not intersect $B$ and contains differences $b-bu^n$, $b\in B$, where $u$ is the universal unitary in $M(B\rtimes_\alpha \Z)$. 
\end{proof}

\begin{cor}\label{cor_1} 
If a pair $(X,J)$ has the uniqueness property, then $X$ is $J$-acyclic.
\end{cor}

\begin{proof}
Assume that $X$ is cyclic with respect to a positively invariant ideal $I\subseteq J$. Then by Lemmas~\ref{lem:cyclicy_implies_equivalence} and~\ref{lem:periodicity_implies_non-uniqueness}, the pair $(I, IX)$ does not have the uniqueness property. Hence $(X,J)$ does not have the uniqueness property by Proposition~\ref{prop:inducing non-detections}.
\end{proof}

Now we turn to the proof of implication \eqref{en:acyclicity}$\Rightarrow$\textup{(3)} in Theorem~\ref{thm:acyclicity}.

\begin{lem}\label{lem:Eduard's theorem}
Let $X$ and $Y$ be two $C^*$-correspondences over $A$ and $B$ respectively. Suppose that $X\sim_M Y$ are Morita equivalent.
\begin{itemize}
\item[(i)] The dual homeomorphism $\widehat{M}:\SB \to \SA$ intertwines the dual multivalued maps $\widehat{X}$ and $\widehat{Y}$.
\item[(ii)] If $A$ is of Type I, then $B$ is of Type I, and the homeomorphism $\widehat{M}:\SB \to \SA$ establishes isomorphism between the dual graphs $E_X$ and $E_Y$, i.e. $m_{[\pi],[\rho]}^Y=m_{\widehat{M}[\pi],\widehat{M}[\rho]}^X$ for all $[\pi], [\rho]\in \SB$.
\end{itemize}
\end{lem}

\begin{proof} We will only show (ii). The reader will easily modify the prove to get (i).
Suppose that $A$ is of Type I. Then $B$ is of Type I because it is Morita equivalent to $A$, cf. \cite{HRW}. Let $[\pi], [\rho]\in \SB$. Recall that $m_{[\pi],[\rho]}^Y$ is the largest cardinal such that there are mutually orthogonal representations $\pi^i$, $i\in I$, $|I|=m_{[\pi],[\rho]}^Y$, with $[\pi^i]=[\pi]$ and $\oplus_{i\in I}\pi^i\leq Y\dashind (\rho)$. Since tensor product of $C^*$-correspondences preservers direct sums, cf. \cite[Proposition 2.69]{morita}, we get that $M\dashind(\pi^i)$, $i\in I$, are mutually orthogonal subrepresentations of $M\otimes Y\dashind (\rho)$ with $\widehat{M}([\pi^i])=[M\dashind(\pi^i)]=[M\dashind(\pi)]=\widehat{M}([\pi])$ for $i\in I$. Since $M\otimes Y\cong X\otimes M$ we have $M\otimes Y\dashind (\rho)\cong X\otimes M\dashind (\rho)$. This implies that $m_{[\pi],[\rho]}^Y\leq m_{\widehat{M}([\pi]),\widehat{M}([\rho])}^X$. By symmetry and the Cantor-Bernstein Theorem we get $m_{[\pi],[\rho]}^Y=m_{\widehat{M}([\pi]),\widehat{M}([\rho])}^X$.
\end{proof}

\begin{lem}\label{cyclic implies equivalence} 
Let $A$ and $B$ be commutative $C^*$-algebras, and let $X$ be a non-degenerate $C^*$-correspondence from $A$ to $B$, that is $AX=X$. Suppose that the dual graph $E_X=(E^1_X, r,s)$ from $\SB$ to $\SA$ is a ``bijection'', i.e. both $r:E^1_X\to \SA$ and $s:E^1_X\to \SB$ are bijections. Then $X$ is an equivalence bimodule (and $A$ is isomorphic to $B$). 
\end{lem}

\begin{proof} By assumption we have $\widehat{B}= \widehat{\langle X,X\rangle_B}$. Hence $B=\langle X,X\rangle_B$, i.e. $X$ is full on the right. Thus it suffices to show that the left action homomorphism $\phi_X:A\to \LL(X)$ is in fact an isomorphism $\phi_X:A\to \K(X)$, cf. \cite[3.3]{ka2} or \cite[Proposition 1.11]{kwa-doplicher}. By Lemma~\ref{range of duals to homomorphisms}, $\widehat{\ker\phi_X}\subseteq \widehat{A}\setminus \overline{r(E^1_X)}$, but by assumption $r(E^1_X)=\SA$. Hence $\phi_X$ is injective.

Since $X$ establishes a Morita equivalence between $\K(X)$ and $B$ we see that (up to unitary equivalence) every representation $\pi$ of $\K(X)$ is of the form $\pi=X\dashind (\sigma)$ for some $[\sigma]\in \widehat{B}$. Due to our assumptions, $A$ acts on the Hilbert space $X\otimes_\sigma H_\sigma$ in a non-degenerate way and this representation contains exactly one irreducible subrepresentation. Since $A$ is commutative, this implies that $X\otimes_\sigma H_\sigma$ is in fact one-dimensional. Hence $\K(X)$ is commutative (as all of its irreducible representations are one-dimensional). In fact we may identify $\K(X)$ with $C_0(\widehat{B})$ and $A$ with $C_0(\widehat{A})$. Then the left action of $A$ on $X$ becomes an injective homomorphism 
$$
\phi_X:C_0(\widehat{A})\to M(C_0(\widehat{B}))=C(\beta(\widehat{B}))
$$
where $\beta(\widehat{B})$ is the Stone-Cech compactification of $\SB$. We need to show that $\phi_X(C_0(\widehat{A}))=C_0(\widehat{B})$. If we assume that $b\in \phi_X(C_0(\widehat{A}))\setminus C_0(\widehat{B})$ then there is $t_0\in \beta(\widehat{B})$ such that $b(t_0)\neq 0$ and $b|_{\phi_X(C_0(\widehat{A}))}\equiv 0$. Such $t_0$ yields an irreducible representation of $C_0(\widehat{A})$ which contradicts surjectivity of $r:E^1_X\to \SA$. Hence $\phi_X(C_0(\widehat{A}))\subseteq C_0(\widehat{B})$. Since two different points $t_1$, $t_2$ in $\SB$, yield two different points $X\dashind (t_1)$ and $X\dashind (t_2)$ in $\SA$, we see that $\phi_X(C_0(\widehat{A}))$ separates the points $\widehat{B}$. Hence $\phi_X(C_0(\widehat{A}))=C_0(\widehat{B})$ by the Stone-Weierstrass Theorem.
\end{proof}

\begin{prop}\label{tf_liminal}
Let $X$ be a $C^*$-correspondence over a $C^*$-algebra $A$, and let $J$ be an ideal of $J_X$. Suppose that $X(J)$ is liminal and $X$ is not topologically free on $J$. Then $X$ is not $J$-acyclic.
\end{prop}

\begin{proof}
Let $E_{Y}$ be the graph dual to the $C^*$-correspondence $Y=JXJ$ over $J$. Since $X$ is not topologically free on $J$, it follows that there exists an ideal $F_0$ in $J$ and $m\in \N$ such that every point in $\widehat{F}_0$ is a base point of a cycle in $E_Y$ of length $m$ without entrances in $E_Y$, and $\X^{-n}(\widehat{F}_0)\subseteq \widehat{J}$ for every $n\in \N$. Since $X(F_0)\subseteq X(J)$ is liminal and $F_0\subseteq J\subseteq J(X)$, we may apply Lemma~\ref{range of duals to correspondences}(ii) to $F_0$. Thus we get $\X^{-1}(\widehat{F}_0)=\widehat{X(F_0)}$, and hence $F_1:=X(F_0)$ is an ideal in $J$. Applying this argument inductively we conclude that the ideals $F_i:=X^i(F_0)$, $i\in \N$, are contained in $J$ and $\X^{-i}(\widehat{F}_0)=\widehat{X^i(F_0)}$ for $i\in \N$. In fact, since $\widehat{F}_0$ consists of base points of length $m$ cycles without entrances, we have $\X^{-m}(\widehat{F}_0)=\widehat{F}_0$, and hence $F_{i (mod\,\,m)}=X(F_{i-1})$ for $i=1,...,m$. In particular, the ideal $I:=F_0+ F_1+...+ F_{m-1}$ is contained in $J$ and $X(I)=I$, so $I$ is positively invariant. The set
$$
\widehat{I}=\bigcup_{i=0}^{m-1}\widehat{F}_i=\bigcup_{i=0}^{m-1} \X^{-i}(\widehat{F}_0)
$$
consists of all base points of cycles attached to points in $\widehat{F}_0$. Since the cycle of length $m$ in $E_Y$ attached to each point in $\widehat{I}$ is unique, we get that both $r:E_{IX}\to \widehat{I}$ and $s:E_{IX}\to \widehat{I} $ are bijections (the graph $E_{IX}$ dual to $IX=IXI$ may be treated as the restriction of $E_Y$ to $\widehat{I}$). We wish to apply Lemma~\ref{cyclic implies equivalence} to the $C^*$-correspondence $IX$. To this end, we claim that by taking a smaller ideal than $F_0$, we may assume that $IX$ is Morita equivalent to a Hilbert bimodule ${_\alpha}Y$ associated to an automorphism on a commutative $C^*$-algebra.

Indeed, since $A$ contains an essential ideal of Type I$_0$, cf. \cite[Theorem 6.2.11]{pedersen}, we may assume by passing to a smaller ideal that $F_0$ is Morita equivalent to a commutative $C^*$-algebra, cf. \cite[Theorem 3.3]{HKS}. For the inductive step, assume that $F_0+ F_1+...+ F_{k-1}$ is Morita equivalent to a commutative $C^*$-algebra for some $k<m$.

1) If $F_{i_0} \cap F_k\neq\{0\}$ for some $i_0=0,...,k-1$, then we may put $F_{i}':=X^{i+m-k}(F_{i_0} \cap F_k)$, for $i=0,...,k$. Then $0\neq F_i'\triangleleft X^i(X^{m-k}(F_k))=X^i(F_0)=F_i$ for $i=0,...,k-1$ and $F_k'=X^{m}(F_{i_0} \cap F_k)=F_{i_0} \cap F_k \triangleleft F_{i_0}$. Hence $F_0'+ F_1'+...+ F_{k}'\triangleleft F_0+ F_1+...+ F_{k-1}$ is Morita equivalent to a commutative $C^*$-algebra.

2) Assume that $F_i \cap F_k=\{0\}$ for all $i_0=0,...,k-1$. Take any non-zero ideal $F_k'$ in $F_k$ which is Morita equivalent to a commutative $C^*$-algebra. Then $F_0+ F_1+...+ F_{k-1} +F_k'$	is Morita equivalent to a commutative $C^*$-algebra (as a direct sum of algebras with this property). Moreover, putting $F'_{i}:=X^{i+m-k}(F_k') \subseteq F_i$ for $i=0,...,k-1$. We get that $F_0'+ F_1'+...+ F_{k}' \triangleleft F_0+ F_1+...+ F_{k-1} +F_k'$ is Morita equivalent to a commutative $C^*$-algebra, and $F_i'=X^i(F_0')$, $i=0,...,k$.

This, by induction, proves our claim. Thus we assume that $I$ is Morita equivalent to a commutative algebra.

Let $M$ be an equivalence bimodule establishing Morita equivalence from a commutative $C^*$-algebra $C_0(V)$ to the ideal $I$. Then $Y:=M\otimes_I IX \otimes_I M^*$ is a $C^*$-correspondence over $C_0(V)$ which is Morita equivalent to $IX$. The graphs dual to $Y$ and $IX$ are equivalent, by Lemma~\ref{lem:Eduard's theorem}. Thus $Y$ is an equivalence bimodule by Lemma~\ref{cyclic implies equivalence}. The structure of an equivalence bimodule over $C_0(V)$ is well known. Namely, there is a homeomorphism $\theta:V\to V$ induced by $Y$ (it is equal to $r\circ s^{-1}$ where $(E_Y,r,s)$ is a graph dual to $Y$). There is a line bundle $L$ over $V$ such that $Y$ is isomorphic to the space of sections of $L$, with $C_0(V)$ acting by pointwise multiplication on the right and by pointwise multiplication composed with $\theta^*$ on the left. In our case we also have that $\theta^m=id$. Since every line bundle is locally trivial, using the same inductive argument as in the proof of the claim above, we may find a non-empty open set $U\subseteq V$ such that $\theta(U)=U$ and the restricted Hilbert bimodule $C_0(U)Y=C_0(U)YC_0(U)$ is isomorphic to the canonical equivalence bimodule associated to the homeomorphism $\theta:U\to U$. Then it is straightforward to see that $(C_0(U)Y)^{\otimes m}\cong C_0(U)$. Since $C_0(U)$ is positively $Y$-invariant, the ideal $I':=M(C_0(U))$ is positively $X$-invariant ideal in $I$:
\begin{align*}
I'X&=I' IX\cong I' (M^*\otimes Y \otimes M) = M^*\otimes C_0(U)Y \otimes M \\
&= M^*\otimes C_0(U)YC_0(U) \otimes M \subseteq (M^*\otimes Y \otimes M) I' \subseteq X I'.
\end{align*}
Moreover, using that $(C_0(U)Y)^{\otimes m}\cong C_0(U)$ we get
\begin{align*}
(I'X)^{\otimes m}&=(M^*\otimes C_0(U)Y \otimes M)^{\otimes m}\cong M^*\otimes ( C_0(U)Y )^{\otimes m} \otimes M \\
&\cong M^*\otimes C_0(U)\otimes M \cong I'.
\end{align*}
Thus, $X$ is cyclic with respect to $I'$. 
\end{proof}

Theorem~\ref{thm:acyclicity} follows from Corollary~\ref{cor_1}, Proposition~\ref{tf_liminal} and Theorem~\ref{uniqueness theorem}.

\section{Applications and examples}\label{sec:Applications}

\subsection{Toeplitz algebras.}
Let $X$ be a $C^*$-correspondence over a $C^*$-algebra $A$. \emph{Toeplitz algebra of $X$}, denoted by $\mathcal{T}_X$, is the $C^*$-algebra generated by the universal representation of $X$. In other words, it is the relative Cuntz--Pimsner algebra relative the zero ideal , that is $\mathcal{T}_X=\OO(\{0\},X)$. Fowler and Raeburn showed in \cite[Theorem 2.1]{Fow-Rae} that if $(\psi_0,\psi_1)$ is a representation of $X$ on Hilbert space $H$ that satisfies the following geometric condition:
\begin{equation}\label{Coburn condition666}
\psi_0\,\, \text{ acts faithfully on }(\psi_1(X)H)^\bot, 
\end{equation}
then $C^*(\psi)\cong \mathcal{T}_X$ with the isomorphism given by $\psi_0\rtimes_0 \psi_1$. This result was later generalized to product systems and more general structures, see \cite{F99}, \cite{kwa-larI}, \cite{kwa-larII}. One of the reasons for further development is that condition \eqref{Coburn condition666} when applied directly to a $C^*$-correspondence, whose left action is not by compacts, is too strong - it is not equivalent to faithfulness of $\psi_0\rtimes_0 \psi_1$. For instance, one can not deduce from \cite[Theorem 2.1]{Fow-Rae} simplicity of the Cuntz algebra $\OO_\infty$ viewed as the Toeplitz algebra $\mathcal{T}_X$ associated to an infinite dimensional Hilbert space $X$.

Our uniqueness theorem shows that the following weaker, algebraic condition:
\begin{equation}\label{Toeplitz condition666}
\forall_{a\in A}\,\,\, \psi_0(a)\in \psi^{(1)}(\K(X)) \,\,\Longrightarrow\,\, a=0; 
\end{equation}
is in fact equivalent to faithfulness of $\psi_0\rtimes_0 \psi_1$:

\begin{thm}[Uniqueness theorem for Toeplitz algebras]\label{theorem for Toeplitz algebras}
Let $X$ be an arbitrary $C^*$-correspondence over a $C^*$-algebra $A$ and let $(\psi_0,\psi_1)$ be a representation of $X$. We have $C^*(\psi)\cong \mathcal{T}_X$, with the isomorphism determined by $\psi_0(a)\mapsto j_A(a)$ and $\psi_1(x)\mapsto j_X(x)$, if and only if $(\psi_0,\psi_1)$ satisfies \eqref{Toeplitz condition666}.
\end{thm}

\begin{proof}
Clearly, $(\psi_0,\psi_1)$ satisfies \eqref{Toeplitz condition666} if and only if $\psi_0$ is injective and \eqref{eq:strict_J__covariance0} holds with $J=\{0\}$. Moreover, the assumptions of Theorem~\ref{uniqueness theorem} with $J=\{0\}$ are trivially satisfied. 
\end{proof}

The above result supports the conjecture discussed below \cite[Theorem 2.19]{kwa-larII}.

\subsection{Simple Cuntz--Pimsner algebras.}\label{sub:Simple Cuntz-Pimsner algebras}

Let $X$ be a $C^*$-correspondence over a $C^*$-algebra $A$. We present here a number of simplicity criteria for $\OO_X$. We start by noting that a relative Cuntz--Pimsner algebra is not simple unless it is the unrelative Cuntz--Pimsner algebra $\OO_X$. Moreover, the class of $C^*$-correspondences that yield simple $\OO_X$ divides into two subclasses where either the left action of $A$ on $X$ is injective, or $X$ is \emph{quasi-nilpotent}, i.e. $\lim_{n\to \infty}\|\phi_{X^{ \otimes n}}(a)\|=0$ for every $a\in A$. We say that \emph{$X$ is minimal} if there are no non-trivial $X$-invariant ideals in $A$.

\begin{lem}\label{dichotomy lemma}
Suppose that $J\subseteq J_X$ and $\OO(J,X)$ is simple. Then necessarily $J=J_X$; $X$ is minimal; and either $\phi_X$ is injective, or $X$ is quasi-nilpotent and $J(X)=A$.

Moreover, $\OO_X$ is simple if and only if $X$ is minimal and has the uniqueness property.
\end{lem}

\begin{proof}
As $(\{0\}, J)$ and $(\{0\}, (\ker\phi_X)^\bot)$ are always $T$-pairs of $X$, we see by \cite[Theorem 11.9]{ka3} that $\OO(J,X)$ is not simple unless $J=J_X$. Similarly, $\OO_X$ is not simple unless there are no $X$-invariant ideals in $A$ (for any such ideal $I$ the pair $(I,I+J_X)$ is a $T$-pair). Suppose now that $\phi_X$ is not injective. Then $I_0=\ker\phi_X$ is a non-zero $X$-invariant ideal in $A$. We define $I_1=I_0+ J_X\cap X^{-1}(I_0)$ and $I_{n}:= I_{n-1} + X^{-1}(I_{n-1})$ for $n\geq 1$. Then $I_\infty:=\lim_{n\to \infty} I_n$ is $X$-invariant ideal by \cite[Lemma 4.15 and Proposition 4.16]{ka3}. Thus $I_\infty=A$ by minimality of $A$. For every $a\in I_n$, $n\in \N$, we have $a\in J(X)$ and $ \|\phi_{X^{ \otimes n}}(a)\|=0$. This implies that $X$ is quasi-nilpotent and $J(X)=A$.

Suppose that $X$ is minimal. Then the only $T$-pairs $(I,I')$ for $X$ with $J_X\subseteq I'$ are $(0,J_X)$ and $(A,A)$. Hence $\OO_X$ contains no non-trivial gauge invariant ideals. Thus, cf. Lemma~\ref{lemma on intersection property}(v), $\OO_X$ is simple if and only if $X$ has the uniqueness property.
\end{proof}

\begin{cor}\label{cor:simplicity1}
If the dual multivalued map $\widehat{X}$ is weakly topologically aperiodic and $X$ is minimal, then $\OO_X$ is simple.
\end{cor}

\begin{proof}
Combine Theorem~\ref{uniqueness theorem} and the second part of Lemma~\ref{dichotomy lemma}.
\end{proof}

When $X$ is a Hilbert $A$-bimodule, then $\widehat{X}$ is a partial homeomorphism. Thus in this case, it follows from \cite[Theorem 9.12]{KM} that the implication in Corollary~\ref{cor:simplicity1} is in fact an equivalence (at least when $A$ contains an essential ideal which is separable or of Type I). Moreover, for any quasi-nilpotent $C^*$-correspondence $X$ the dual map $\X$ has no ``periodic points''. Indeed, if there is $[\pi]\in \SA$ and $n>0$ such that $X^{\otimes n}\dashind (\pi) \cong\pi$, then for any $a\in A$ with $\|\pi(a)\| \geq 1$ and any $m\in \N$ we get that $\|\phi_{X^{\otimes mn}}(a)\|\geq \|X^{\otimes mn}\dashind (\pi)(a)\|=\|\pi(a)\| \geq 1$. Thus when the left action on $X$ is not injective we have the following characterization of simplicity of $\OO_X$ in terms of aperiodicity of $\X$:

\begin{prop}\label{simple:non-injective}
Let $A$ be a $C^*$-algebra and $X$ a Hilbert $A$-bimodule. Suppose that the left action of $A$ on $X$ is not injective. The following statements are equivalent:
\begin{enumerate}
\item $\OO_X$ is simple; 
\item $X$ is minimal and quasi-nilpotent;
\item $X$ is minimal and the dual multivalued map $\X$ is topologically aperiodic.
\end{enumerate}
\end{prop}

\begin{proof} 
We have (1)$\Rightarrow$(2) by Lemma~\ref{dichotomy lemma}, and (2)$\Rightarrow$(3) is clear by the discussion above. Implication (3)$\Rightarrow$(1) follows from Corollary~\ref{cor:simplicity1}.
\end{proof}

When the left action is injective we can use the main result of \cite{Schweizer} to get the following characterization of simplicity in terms of acyclicity.

\begin{prop}
Let $X$ be a $C^*$-correspondence over a unital $C^*$-algebra $A$. Suppose also that $X$ is full on the right and the left action of $A$ on $X$ is injective and $AX=X$. The following statements are equivalent:
\begin{enumerate}
\item $\OO_X$ is simple;
\item $X$ is minimal and non-periodic;
\item $X$ is minimal and $A$-acyclic.
\end{enumerate}
Moreover, if $J(X)\neq A$, then $\OO_X$ is simple if and only if $X$ is minimal.
\end{prop}

\begin{proof} 
We have (1)$\Leftrightarrow$(2) by \cite[Theorem 3.9]{Schweizer}. To see that (2)$\Leftrightarrow$(3) it suffices to note that our assumptions and minimality of $X$ imply that there are no non-trivial positively $X$-invariant ideals in $A$. Indeed, suppose that $X$ is minimal and $I_0$ is a positively $X$-invariant ideal in $A$ and $I_0\neq A$. Then $I_1:=X^{-1}(I)=\{a\in A: \langle x,a\cdot y\rangle_A \in I \textrm{ for all } x,y\in X \}$ contains $I_0$ but does not contain $1$. Proceeding in this way we get an increasing sequence $\{I_n\}_{n=1}^\infty$ of ideals in $A$, each of which does not contain $1$. The ideal $I_\infty= \overline{\bigcup_{n=1}^\infty I_n}$ is not equal to $A$ (does not contain $1$) and it is $X$-invariant, cf. \cite[Proposition 4.16]{ka3}. Thus minimality of $X$ implies that $I_\infty=\{0\}$ and hence $I_0=\{0\}$.

If $J(X)\neq A$, then $X$ cannot be an equivalence bimodule and all the more it cannot be periodic. 
\end{proof}

Finally, in the liminal case, using our main results we get:

\begin{prop}
Let $X$ be a $C^*$-correspondence over a $C^*$-algebra $A$ such that $X(J_X)$ is liminal. The following statements are equivalent:
\begin{enumerate}
\item $\OO_X$ is simple;
\item $X$ is minimal and the graph dual to $X$ is topologically free on $\widehat{J_X}$;
\item $X$ is minimal and $J_X$-acyclic.
\end{enumerate}
\end{prop}

\begin{proof} 
Combine the last parts of Theorem~\ref{thm:acyclicity} and Lemma~\ref{dichotomy lemma}.
\end{proof}

\subsection{Topological correspondences.}\label{sub:Topological correspondences} 

Let us fix a quadruple $E=(E^0,E^1,r,s)$ consisting of locally compact Hausdorff spaces $E^0$, $E^1$ and continuous maps $r,s:E^{1}\to E^0$. We will refer to $E$ as to a \emph{topological graph}. We also fix a \emph{continuous family of measures along fibers of $s$}, i.e. a family of Radon measures $\lambda=\{\lambda_v\}_{v\in E^0}$ on $E^1$ such that
\begin{itemize}
\item[(Q1)] $\supp \lambda_v\subseteq s^{-1}(v)$ for all $v \in E^0$,
\item[(Q2)] $v \to \int_{E^1} a(e) d\lambda_{v} (e)$ is an element of $C_c(E^0)$ for all $a\in C_c(E^1)$.
\end{itemize}
In \cite{BHM}, the quintuple $\QQ=(E^0,E^1,r,s,\lambda)$ is called a \emph{topological correspondence} (from $E^0$ to $E^0$). It is a mild but important generalization of a topological quiver introduced in \cite[Example 5.4]{ms} and studied in \cite{mt}. Namely, a \emph{topological quiver} is a topological correspondence where in (Q1) we have $\supp \lambda_v= s^{-1}(v)$ for all $v \in E^0$. We note that if $E=(E^0,E^1,r,s)$ is a topological graph in the sense of Katsura \cite{ka1}, i.e. if we additionally assume that $s$ is a local homeomorphism, then each set $s^{-1}(v)$ is discrete and we may treat $E$ as a topological quiver equipped with the family $\lambda=\{\lambda_v\}_{v\in E^0}$ of counting measures on $s^{-1}(v)$, $v\in E^0$. Conversely, if $\QQ$ is a topological quiver such that each $\lambda_v$ is a counting measure on $s^{-1}(v)$, then $s$ is necessarily a local homeomorphism.

We define the \emph{support of the family of measures} $\lambda$ as the following union:
$$
\supp \lambda :=\bigcup_{v\in E^0} \supp \lambda_v.
$$
Axiom (Q2) implies that the set map $v \mapsto \supp \lambda_v$ is lower semicontinuous, cf. \cite[Lemma 3.28]{kwa-exel}. This together with axiom (Q1) gives that the restricted source map $s:\supp \lambda \to E^0$ is open. Note that the topological correspondence $\QQ$ is a topological quiver if and only if $\supp \lambda=E^{1}$. Moreover, the $C^*$-correspondence we associate to $\QQ$ depends only on the closure of $\supp \lambda$ in $E^{1}$. Thus without loss of generality we could assume that $\overline{\supp \lambda}=E^{1}$.

We define the \emph{$C^*$-correspondence $X_\QQ$ associated to} $Q$ as the Hausdorff completion of the semi-inner $C^*$-correspondence over $A=C_0(E^0)$ defined on $C_c(E^1)$ via
$$
(a \cdot f \cdot b) =(a \circ r) f (b\circ s) \quad \textrm{ and }\quad\langle f , g \rangle_{ C_0(E^0)} (v)=\int_{s^{-1}(v)} \overline{f} g d\lambda_v,
$$
$f,g \in C_c(E^1)$, $a,b \in C_0(E^0)$, see \cite[Definition 2.5]{BHM}, \cite[3.1]{mt}. We define the \emph{quiver $C^*$-algebra} associated to $\QQ$ as $\OO_{X_\QQ}$, cf. \cite[Definition 3.17]{mt}. For every $V\subseteq \widehat{J(X)}$ we define the corresponding \emph{relative quiver $C^*$-algebra} as $\OO(C_0(V),X_\QQ)$, cf. \cite[Section 7]{mt}.

We wish to compare the topological graph $E$ with the graph $E_\QQ$ dual to the $C^*$-correspondence $X_\QQ$.

\begin{lem}\label{lemma:dual_graph_to_topological_quiver}
The multiplicity of a pair of vertices $(w,v)\in E^0\times E^0$ for the graph dual $E_\QQ$ to $X_\QQ$ is the dimension of the Hilbert space $L^2( r^{-1}(v)\cap s^{-1}(w), \lambda_w)$:
$$
m_{w,v}^\QQ=\dim (L^2( r^{-1}(w)\cap s^{-1}(v), \lambda_v)).
$$
\end{lem}

\begin{proof}
We identify points of $E^0$ with irreducible representations of $C_0(E^0)$ by putting $v(a):=a(v)$ for $v\in E^0$, $a\in C_0(E^0)$. Let $v\in E^0$ and consider a representation $\pi_v:C_0(E^0)\to \B(L^2_{\lambda_v}(s^{-1}(v)))$ defined by
$$
(\pi_v(a)f)(e)=a(r(e))f(e), \qquad a \in C_0(E^0), \,\, f\in L^2_{\lambda_v}(s^{-1}(v)). 
$$ 
Then the formula $ \big(U( x\otimes_{v}\lambda )\big)(e):=\lambda x(e)$, $e\in s^{-1}(v),\,\, x\in X_\QQ,\,\,\lambda\in \C, $ defines a unitary establishing the equivalence $X_\QQ\dashind (v) \cong \pi_v$, see \cite[Lemma 2.3]{BHM}. Plainly, we have $w\leq \pi_v$ if and only if $L^2( r^{-1}(v)\cap s^{-1}(w), \lambda_w)$ is a non-zero subspace of $L^2_{\lambda_v}(s^{-1}(v))$, and the multiplicity of the subrepresentation $w$ of $\pi_v$ is equal to the dimension of $L^2( r^{-1}(v)\cap s^{-1}(w), \lambda_w)$. This implies the assertion.
\end{proof}

The above lemma implies that if $E=(E^0,E^1,r,s)$ is a topological graph in the sense of Katsura, equipped with counting measures, then the graph dual $E_\QQ$ to $X_\QQ$ is equivalent to $E$. For general topological correspondences the graph $E_\QQ$ is equivalent to a proper subgraph of $E$.

\begin{ex}\label{ex:topological correspondences without multiplicities} 
Let $\QQ=(E^0,E^1,r,s,\lambda)$ be a topological correspondence such that the graph $E=(E^0,E^1,r,s)$ has no multiple edges. It follows from Lemma~\ref{lemma:dual_graph_to_topological_quiver} that the graph $E_\QQ$ dual to $X_\QQ$ is a subgraph of $E$ where $E_\QQ^1$ consists of edges $e\in E^1$ with $\lambda_{s(e)}(\{e\}) >0$. Thus for instance if we put $E^0=[0,1]$, $E^1=[0,1]\times [0,1]$, $s(x,y)=x$, $r(x,y)= y$ and let $\lambda_x$ be the Lebesgue measure on $[0,1]$, for $x,y \in [0,1]$. Then $(E^0,E^1,r,s, \{\lambda_{x}\}_{x\in E^0})$ is a topological quiver and the graph $E_\QQ$ has no edges.
\end{ex}

Even though, as we have seen above, the graph $E_\QQ$ might have much fewer edges than $E$, this difference does not affect topological freeness (at least for topological quivers).

\begin{prop}\label{prop:topological freeness for quivers}
Let $\QQ=(E^0,E^1,r,s, \{\lambda_{v}\}_{v\in E^0})$ be a topological quiver. The graph $E=(E^0,E^1,r,s)$ is topologically free if and only if the graph $E_\QQ$ dual to $X_\QQ$ is topologically free on $\widehat{J_X}$.
\end{prop}

\begin{proof} By \cite[Proposition 3.15]{mt}, we have $E^0_{fin}=\widehat{J(X_\QQ)}$ where 
\begin{align*}
E^0_{fin}:=\{v \in E^0:\,\,\, & \textrm{there exists a neighborhood } U \textrm{ of }v\textrm{ such that } \\
& r^{-1}(\overline{U}) \textrm{ is compact and }s|_{r^{-1}(U)} \textrm{ is a local homeomorphism} \}.
\end{align*}
Thus, for every $w,v\in E^0$ with $w\in \widehat{J(X)}$ the measure $\lambda_v$ restricted to the set $r^{-1}(w)$ is discrete (is equivalent to a counting measure on $r^{-1}(w)\cap \supp \lambda_v$). In view of Lemma~\ref{lemma:dual_graph_to_topological_quiver}, this implies that the number of edges from $v$ to $w$ in graphs $E_\QQ$ and $E$ is the same. In other words, the restricted graphs ${_{\widehat{J(X)}}}E_{\QQ}$ and ${_{\widehat{J(X)}}}E$ are equivalent. Moreover, we have $ \widehat{J_X}= E^0_{fin}\cap \Int(\overline{r(E^1)})$, cf. \cite[Proposition 3.15]{mt}.

Now, assume that $E_\QQ$ is not topologically free on $\widehat{J_X}$. Let $V\subseteq \widehat{J_X}$ be an open non-empty set consisting of base points of cycles in ${_{\widehat{J_X}}}(E_\QQ)_{\widehat{J_X}}$ of length $n$ without entrances in $E_\QQ$. Since ${_{\widehat{J(X)}}}E_\QQ$ and ${_{\widehat{J(X)}}}E$ are equivalent, this implies that $V$ consists of base points of cycles in $E$ without entrances in $E$. Hence $E$ is not topologically free.

Conversely, assume that $E$ is not topologically free. Let $V_0\subseteq E^0$ be an open non-empty set consisting of base points of cycles in $E^n$ without entrances in $E$. Defining inductively $V_k=s(r^{-1}(V_{k-1}))$ for $k=1,...,n-1$, we see that $V:=\bigcup_{k=0}^{n-1} V_k$ is an open set consisting of base points of cycles in $({_VE}_V)^n$ without entrances in $E$. Note that $s:s^{-1}(V)\to V$ is a homeomorphism (a continuous open bijection). It follows that the map $h:V\to V$ given by $h:=r \circ s^{-1}$ is continuous and bijective. Moreover, we have $h^n=id$ and in particular $h^{-1}=h^{n-1}$ is continuous. Hence $h$ is a homeomorphism. Since $r:r^{-1}(V)\to V$ is equal to the composition $h\circ s$ of homeomorphisms, this map is a homeomorphism. Thus we see that $V\subseteq \widehat{J_X}= E^0_{fin}\cap \Int(\overline{r(E^1)})$. Accordingly, the graph $E_\QQ$ is not topologically free on $\widehat{J_X}$.
\end{proof}

\begin{rem}
It seems very likely that the preceding proposition holds for every topological correspondence with $\overline{\supp \lambda}=E^{1}$. The proof would require description of the spectrum of $J(X)$, which in case of topological quivers is provided by \cite[Proposition 3.15]{mt}.
\end{rem}

Suppose that $\QQ=(E^0,E^1,r,s, \{\lambda_{v}\}_{v\in E^0})$ is a topological quiver. It is shown in \cite[Theorem 6.16 and Corollary 7.16]{mt} that topological freeness of $E=(E^0,E^1,r,s)$ implies the uniqueness property of every pair $(C_0(V),X)$ with $V\subseteq \widehat{J(X)}$. We extend this result to topological correspondences. In addition we get that topological freeness of $E$ is not only sufficient but also necessary for the uniqueness property for $(X,J_X)$. For relative Cuntz--Pimsner algebras we improve upon results of \cite{mt} by showing that a weaker form of topological freeness of $E$ is sufficient (in fact equivalent) to the uniqueness property for $\OO(C_0(V),X)$ with $V\subseteq \widehat{J_X}$. The following result generalizes also Katsura's uniqueness theorems \cite[Theorem 5.12]{ka1}, \cite[Theorem 6.14]{ka4}.

\begin{thm}\label{thm:Cuntz-Krieger uniqueness for quivers} 
Let $\QQ=(E^0,E^1,r,s, \{\lambda_{v}\}_{v\in E^0})$ be a topological correspondence and let $V\subseteq \widehat{J_{X_\QQ}}$ be an open set. Let $E_\QQ$ be a subgraph of $E$ with multiplicities given in Lemma~\ref{lemma:dual_graph_to_topological_quiver}. Every injective representation of $X_\QQ$ satisfying \eqref{eq:strict_J__covariance} integrates to a faithful representation of $\OO(C_0(V),X)$ if and only if $E_\QQ$ is topologically free on $V$.
\end{thm}

\begin{proof}
Apply Theorem~\ref{thm:acyclicity}.
\end{proof}

\begin{cor}
If $\QQ=(E^0,E^1,r,s, \{\lambda_{v}\}_{v\in E^0})$ is a topological quiver, then the graph $E=(E^0,E^1,r,s)$ is topologically free if and only if every injective covariant representation of $X_\QQ$ integrates to a faithful representation of the quiver $C^*$-algebra $\OO_{X_\QQ}$.
\end{cor}

\begin{proof}
Combine Theorem~\ref{thm:Cuntz-Krieger uniqueness for quivers} and Proposition~\ref{prop:topological freeness for quivers}.
\end{proof}

Generalization of the theory of topological quivers to topological correspondences is important for instance in the theory of crossed products by (completely) positive maps:

\begin{ex}[Crossed products by positive maps on $C_0(V)$]\label{ex:positive maps}
Let $P:C_0(V)\to C_0(V)$ be a positive map where $V$ is a locally compact Hausdorff space. When $V$ is compact and $P(1)=1$, such maps are called Markov operators in \cite{imv}. As in \cite[Lemma 3.30]{kwa-exel}, we associate to $V$ a topological correspondence where the space of edges $E^1$ is the closure of a relation $R\subseteq V\times V$ defined by 
$$
(v,w)\in R\, \, \, \stackrel{def}{\Longleftrightarrow} \, \, \, \left(\forall_{a\in C_0(V)_+}\,\, a(w)>0 \, \Longrightarrow\, P(a)(v)> 0\right).
$$
Source and range maps are given by $s(v,w):=v$, $r(v,w):=w$. The system of measures $\lambda=\{\lambda_{v}\}_{v\in E^0}$ is determined by
$$
P(a)(v)=\int_{s^{-1}(w)} a(w) d\lambda_v (w), \qquad v\in V, a\in A.
$$
Then $\QQ=(V,\overline{R},r,s, \{\lambda_{v}\}_{v\in V})$ is a topological correspondences and $R=\supp \lambda$. By \cite[Proposition 3.32 and Theorem 3.13]{kwa-exel}, the crossed product $C^*(C_0(V), P)$ constructed in \cite[Definition 3.5]{kwa-exel} is naturally isomorphic to the quiver $C^*$-algebra $\OO_{X_\QQ}$. One easily finds examples, cf. \cite[Example 3.5]{kwa-exel}, of Markov operators for which the topological correspondence $\QQ$ is not the topological quiver in the sense of \cite{mt}, i.e. $\supp\lambda_v= s^{-1}(v)$ does not hold in general (equivalently $R$ is not closed in $V\times V$). The graph $E=(V,\overline{R},r,s)$ has no multiple edges. Hence the graph $E_\QQ$ dual to $X_\QQ$ arises by taking all atoms of measures $\lambda_v$, $v\in V$, see Example~\ref{ex:topological correspondences without multiplicities}. When $P$ is multiplicative then $s$ is injective and the associated algebras are crossed products by endomorphisms, see Example~\ref{endomorphisms of C(V)-algebras} below. When $P$ is a transfer operator then $r$ is injective and the associated algebras are Exel's crossed products, see the next example.
\end{ex}

\begin{ex}[Exel's crossed products]\label{Exel's crossed products}
Suppose that $\LL:C_0(V)\to C_0(V)$ is a \emph{transfer operator}, that is $\LL$ is positive and there exists an endomorphism $\alpha:C_0(V)\to C_0(V)$ such that $\LL(\alpha(a)b)=a\LL(b)$ for all $a,b \in C_0(V)$. The Exel's crossed product $C_0(V)\rtimes_{\alpha,\LL}\N$ is naturally isomorphic to the crossed product $C^*(C_0(V), \LL)$ by $\LL$, see \cite[Theorem 4.7]{kwa-exel}. Here we consider Exel's crossed products as defined in \cite{er} (the crossed product originally defined in \cite{exel2} coincides with the modified one in a number of natural cases, see \cite{kwa-exel}). Hence $C_0(V)\rtimes_{\alpha,\LL}\N\cong \OO_{X_\QQ}$ where $\QQ$ is the topological correspondence associated to $\LL$ as in Example~\ref{ex:positive maps}. Endomorphism $\alpha$ is given by a composition with a continuous proper map $\varphi:\Delta \to V$ defined on an open set $\Delta\subseteq V$. For each $v \in V$ we have $\supp\lambda_v\subseteq \varphi^{-1}(v)$. Hence the graph $E_\QQ$ can be identified with the graph associated to the map $\varphi$ restricted to the set $\Lambda:=\{ x\in \Delta: \lambda_{\varphi(x)}(\{x\})>0 \}$. By Example~\ref{partial map example}, we conclude that Exel's crossed product $C_0(V)\rtimes_{\alpha,\LL}\N$ has the uniqueness property if and only if $ \varphi:\Lambda \to V\text{ is topologically aperiodic on } V$. This generalizes uniqueness theorems \cite[Theorem 9.1]{exel_vershik}, \cite[Theorem 6]{CS}, \cite[Theorem 6.1]{brv} proved in the case $\varphi$ is a local homeomorphism and $\Lambda=\Delta=V$.
\end{ex}

\subsection{Crossed products by endomorphisms.}\label{sub:Crossed products by endomorphisms} 

We start with recalling a general definition of crossed products by endomorphisms from \cite{kwa-endo}. Such $C^*$-algebras are special cases of crossed products by completely positive maps, see \cite[Proposition 3.26]{kwa-exel}.

Let $\alpha:A\to A$ be an endomorphism of a $C^*$-algebra $A$. A \emph{representation} of an endomorphism $\alpha$ in $C^*$-algebra $B$ is a pair $(\pi,U)$ where $\pi:A\to B$ is non-degenerate homomorphism of $A$ and $U\in M_\ell(B)$ is a left multiplier of $B$ such that
$$
U\pi(a)U^* =\pi(\alpha(a)),\qquad \textrm{ for all }a \in A.
$$
Then $U$ is necessarily a partial isometry and $U^*U\in \pi(A)'$. We call $C^*(\pi, U):=C^*(\pi(A)\cup U\pi(A))\subseteq B$ the \emph{$C^*$-algebra generated by $(\pi,U)$}. Then $U\in M_\ell(C^*(\pi, U))$ is a left multiplier of $C^*(\pi, U)$. If $\alpha$ is extendible, i.e. it extends to a strictly continuous endomorphism of the multiplier algebra $M(A)$, then $U\in M(C^*(\pi, U))$, see \cite[Remark 4.1]{kwa-rever}. Let $J$ be an ideal in $(\ker\alpha)^\bot$. We say that a representation $(\pi,U)$ of $\alpha$ is \emph{$J$-covariant} if $\{a: \pi(a)U^*U=\pi(a)\}\subseteq J$. We say that $(\pi,U)$ is \emph{strictly $J$-covariant} if $\{a: \pi(a)U^*U=\pi(a)\}= J$. The corresponding \emph{relative crossed product} can be defined, see \cite[Definition 2.7]{kwa-endo}, as a $C^*$-algebra $C^*(A,\alpha;J)$ generated by a universal $J$-covariant representation $(\iota_A, u)$. The representation $(\iota_A, u)$ is necessarily strictly $J$-covariant and injective, in the sense that $\iota_A$ is injective. By definition every $J$-covariant representation integrates to a representation of $C^*(A,\alpha;J)$. In the case when $J=(\ker\alpha)^\bot$ we write $ C^*( A,\alpha):=C^*( A,\alpha;(\ker\alpha)^\bot) $ and call it the (unrelative) \emph{crossed product of $A$ by $\alpha$}.

We associate to $\alpha$ a $C^*$-correspondence $X_\alpha$ defined by the formulas: 
$$
X_\alpha:=\alpha(A)A, \,\, \,\, \langle x, y\rangle_A :=x^*y, \,\,\,\, a\cdot x\cdot b:=\alpha(a)xb, \,\,\,\, \,\,x,y\in \alpha(A)A,\,\, a,b \in A.
$$ 
Then $J_{X_\alpha}=(\ker \alpha)^\bot$ and there is a bijective correspondence between $J$-covariant representations of $\alpha$ and of $X_\alpha$, see \cite[Proposition A.8]{kwa-endo}. Thus, for every ideal $J\subseteq (\ker \alpha)^\bot$, we have
$$
C^*( A,\alpha;J)\cong \OO(J,X_\alpha),\qquad C^*( A,\alpha)\cong \OO_{X_\alpha}.
$$
Applying case A1) in our uniqueness theorem and results of \cite{KM} we get the following theorem. The second part extends the known fact that topological freeness of an automorphism on a separable $C^*$-algebra is equivalent to the uniqueness property, cf. \cite[Theorem 6.6]{OlPe}, see also \cite[Corollary 1]{Ortega_Pardo}, \cite[Theorem 4.2]{kwa-rever}.

\begin{thm}\label{thm:uniqueness for crossed products1}
Let $J$ be an ideal in $(\ker \alpha)^\bot$. Consider the following conditions
\begin{enumerate}
\item the multivalued map $\widehat{\alpha}$ is weakly topologically aperiodic on $\widehat{J}$;
\item every injective strictly $J$-covariant representation $(\pi,U)$ of $\alpha$ integrates to an isomorphism $ C^*(\pi,U)\cong C^*( A,\alpha;J)$.
\end{enumerate}
Then (1)$\Rightarrow$(2). If $\alpha$ has a complemented kernel and a hereditary range, $A$ contains an essential ideal which is either separable or of Type I, and $J=(\ker\alpha)^\bot$, then (1)$\Leftrightarrow$(2).
\end{thm}

\begin{proof}
Identifying the spectra $\widehat{A\alpha(A)A}$ and $\widehat{\alpha(A)A\alpha(A)}$ with open subsets of $\widehat{A}$, we have $\widehat{A\alpha(A)A}=\widehat{\alpha(A)A\alpha(A)}$. In particular, composing $X_\alpha\dashind:\widehat{A\alpha(A)A}\to \widehat{\K(X_\alpha)}$ with the homeomorphism dual to the isomorphism $\K(E_\alpha)\ni \Theta_{x,y}\mapsto xy^* \in \alpha(A)A\alpha(A)$, see \cite[Lemma A.7]{kwa-endo}, we get an identity on $\widehat{A\alpha(A)A}$. This implies that the multivalued maps $\widehat{\alpha}$ and $\widehat{X_\alpha}$ coincide. Hence (1)$\Rightarrow$(2) by Theorem~\ref{uniqueness theorem}.

Assume that $J=(\ker\alpha)^\bot$ is complemented ideal in $A$ and $\alpha(A)=\alpha(A)A\alpha(A)$ is a hereditary subalgebra of $A$. Then $X_\alpha$ is a Hilbert bimodule and $\widehat{\alpha}=\widehat{X_\alpha}$ is a partial homeomorphism of $\widehat{A}$. If $A$ contains an essential ideal which is either separable or of Type I, then \cite[Theorem 8.1]{KM} implies that (1)$\Leftrightarrow$(2).
\end{proof}

Suppose now that $K$ is an ideal in $A$ and that $K$ is liminal. We can construct a graph on $\widehat{K}$ as follows. For any $[\rho]\in \widehat{K}$ to $[\pi]\in \widehat{K}$ we define the number $m_{[\pi],[\rho]}$ to be the multiplicity of subrepresentations of $\widetilde{\rho}\circ \alpha|_K$ equivalent to $\pi$, where $\widetilde{\rho}$ is the unique extension of $\rho$ from $K$ to $A$. We denote the corresponding graph by $E_\alpha^K=(\widehat{K}, E^1_\alpha, r,s)$.

\begin{defn}
Let $J$ be an ideal in $(\ker\alpha)^\bot$ and assume that $A\alpha(J)A$ is liminal. Then $K:=J+A\alpha(J)A$ is liminal ($J$ is liminal because it is isomorphic to the subalgebra $\alpha(J)$ of $A\alpha(J)A$). We say that $\alpha$ is \emph{topologically free on $J$} if the graph $E_\alpha^K$ constructed above is topologically free on $\widehat{J}$.
\end{defn}

\begin{defn}
We say that an endomorphism $\alpha:A\to A$ is \emph{inner} if there is an isometry $u\in M_\ell(A)\subseteq A{''}$ such that $\alpha(a)=uau^*$. Otherwise we say that $\alpha$ is \emph{outer}.
\end{defn}

\begin{lem}\label{lem:outerness}
An endomorphism $\alpha:A\to A$ is inner if and only if $X_\alpha$ is isomorphic to the trivial $C^*$-correspondence $A$.
\end{lem}

\begin{proof}
Suppose that $\alpha:A\to A$ is inner and let $u\in M_\ell(A)\subseteq A''$ be an isometry such that $\alpha(a)=uau^*$. We may treat $u$ as the left centralizer $u:A\to A$ of $A$, cf. \cite[3.12.3]{pedersen}. Then $u$ is an isometry of the trivial Hilbert $A$-module $A$. Moreover, $\alpha(A)A=uAu^*A\subseteq uA$ and for every $a,b\in A$ we have $uab=ua u^*u b=\alpha(a)ub\in \alpha(A)A$. Hence $u:A\to X_\alpha=\alpha(A)A$ is a unitary, and since $\alpha(a)ub= uau^*ub=uab$, for $a,b\in A$, we see that $u$ is an isomorphism of $C^*$-correspondences.

Conversely assume that we have an isomorphism of $C^*$-correspondences $u:A\to X_\alpha$. Since $X_\alpha=\alpha(A)A\subseteq A$, we see that $u$ is a left centralizer of $A$ and hence we may treat it is a left multiplier. For an approximate unit in $\{\mu_\lambda\}$ in $A$ and every $a\in A$ we have $a=\lim_{\lambda} \langle \mu_\lambda, a \rangle = \lim_{\lambda} \langle u \mu_\lambda, ua \rangle= \lim_{\lambda} \mu_\lambda u^*u a= u^*ua$. This implies that $u^*u=1\in M_\ell(A)\subseteq A''$. Furthermore, $\alpha(a)=\lim_{\lambda}\alpha(a)\alpha(\mu_\lambda)=\lim_{\lambda}u^*u\alpha(a)\alpha(\mu_\lambda)= \lim_{\lambda} u(a u^*\alpha(\mu_\lambda))=ua u^*$. 
\end{proof}

\begin{thm}\label{thm:uniqueness for crossed products2}
Let $A$ be a $C^*$-algebra and $\alpha:A\to A$ an endomorphism. Let $J$ be an ideal in $(\ker \alpha)^\bot$ such that $A\alpha(J)A$ is liminal. The following conditions are equivalent:
\begin{enumerate}
\item every injective strictly $J$-covariant representation $(\pi,U)$ of $\alpha$ integrates to an
isomorphism $ C^*(\pi,U)\cong C^*( A,\alpha;J)$;
\item $\alpha$ is topologically free on $J$;
\item for every non-zero ideal $I$ in $J$ such that $\alpha(I)\subseteq I$, every power $(\alpha|_{I})^{n}$, $n>0$, of $\alpha|_{I}$ is outer. 
\end{enumerate}
\end{thm}

\begin{proof} Note that $K:=J+A\alpha(J)A$ is liminal ($J$ is liminal because it is isomorphic to the subalgebra $\alpha(J)$ of $A\alpha(J)A$).
Arguing as in the proof of Theorem~\ref{thm:uniqueness for crossed products1} one sees that the graph $E_\alpha^K=(\widehat{K}, E^1_\alpha, r,s)$ is equivalent to the graph dual to the $C^*$-correspondence ${_K}(X_\alpha)_{K}$. Hence $\alpha$ is topologically free on $J$ if and only if $X_\alpha$ is topologically free on $J$. Using that and Lemma~\ref{lem:outerness}, we get the assertion by Theorem~\ref{thm:acyclicity}.
\end{proof}

\begin{ex}[Endomorphisms of $K(H)$]
Let $A=K(H)$ be a $C^*$-algebra of compact operators on a Hilbert space $H$. It is well known that every (non-zero) endomorphism $\alpha:A\to A$ is of the form $\alpha(a)=\sum_{i=1}^n U_i a U_i^*$ where $\{U_i\}_{i=1}^n\subseteq B(H)$ are isometries with mutually orthogonal ranges and $n$ is a natural number called \emph{Power's index of $\alpha$}. It follows that the graph dual to $\alpha$ consists of one vertex and $n$-edges. In particular,
$$
\alpha\text{ is topologically free on } A \,\,\Longleftrightarrow\,\, \text{ Power's index of $\alpha$ is greater than }1.
$$
Thus, by Theorem~\ref{thm:uniqueness for crossed products2}, the crossed product $C^*(A,\alpha)$ has the uniqueness property if and only if $n>1$. Note that the dual multivalued map $\widehat{\alpha}$ is always an identity on the singleton $\SA$. Thus, it is never weakly aperiodic. If $n=1$, then the monomorphism $\alpha$ has a hereditary range and in this case we could infer the lack of uniqueness property for $C^*(A,\alpha)$ from the second part of Theorem~\ref{thm:uniqueness for crossed products1}.
\end{ex}

Let us generalize the previous example.

\begin{ex}[Endomorphisms of $C_0(V,K(H))$]\label{endomorphisms of C(V)-algebras}
Let $V$ be a locally compact Hausdorff space and let $A=C_0(V,K(H))$. Suppose that $(A,\alpha)$ is a $C_0(V)$-dynamical system in the sense of \cite[Definition 3.7]{kwa-endo}, i.e. $\alpha:A\to A$ is an endomorphism of the form
$$
\al(a)(x)=\begin{cases}
\al_x(a(\varphi(x)), & x \in \Delta,
\\
0 & x \notin \Delta,
\end{cases}
$$
where $\varphi:\Delta \to V$ is a continuous proper map defined on an open set $\Delta\subseteq V$, and $\{\al_{x}\}_{x\in \Delta}$ is a continuous bundle of (non-zero) endomorphisms of $K(H)$. The graph $E_\alpha=(\SA, E^1_\alpha, r,s)$ dual to $\alpha$ can be described as follows: there is an edge from $x\in V$ to $y \in V$ if and only if $x\in \Delta$ and $y=\varphi(x)$. Moreover, the multiplicity $m_{\varphi(x),x}^\alpha$, $x\in \Delta$, is equal to Power's index of $\alpha_x$. In particular, the (multivalued) map $\widehat{\alpha}$ coincides with $\varphi$. Let $J=C_0(U)$ be an ideal in $(\ker\alpha)^\bot=C_0(\Delta)$ and put $Y=X\setminus U$. In view of Example~\ref{partial map example} we have 
$$
\widehat{\alpha}\text{ is weakly topologically aperiodic on $\widehat{J}$} \,\,\Longleftrightarrow\,\, \varphi\text{ is topologically free outside $Y$}.
$$
Thus, in this case Theorem~\ref{thm:uniqueness for crossed products1} recovers \cite[Theorem 4.11]{kwa-endo}. In order to improve it let $Z$ be the set of points $x\in \Delta$ for which Power's index of $\alpha_x$ is $1$. One readily sees that 
$$
\alpha\text{ is topologically free on } J \,\,\Longleftrightarrow\,\, \varphi\text{ is topologically free outside $Y\cup Z$}.
$$
Thus, Theorem~\ref{thm:uniqueness for crossed products2} leads to an improvement of \cite[Theorem 4.11]{kwa-endo}. If $\dim(H)<\infty$ then $Z=\emptyset$ and the uniqueness property for $C^*(A,\alpha;J)$ is equivalent to topological freeness of $\varphi$ outside $Y$. This generalizes \cite[Proposition 4.8]{kwa-rever} proved in the commutative case. 
\end{ex}

\end{document}